\documentclass{amsart}
\usepackage{amsthm}
\newtheorem{theorem}{Theorem}[section]
\newtheorem{lemma}[theorem]{Lemma}
\newtheorem{proposition}[theorem]{Proposition}
\theoremstyle{definition}
\newtheorem{definition}[theorem]{Definition}

\theoremstyle{remark}
\newtheorem{remark}[theorem]{Remark}
\newtheorem{example}[theorem]{Example}

\newtheorem{corollary}[theorem]{Corollary}

\usepackage{amssymb}
\usepackage{empheq}
\usepackage{stmaryrd}
\usepackage{enumerate}
\usepackage[shortlabels]{enumitem}
\usepackage{calc}
\usepackage{url}
\usepackage{xcolor}

\usepackage{mathtools}

\usepackage{tikz}
\usetikzlibrary{automata,positioning}

\newcommand{\bdfn}{\begin{definition}}
\newcommand{\edfn}{\end{definition}}
\newcommand{\bthm}{\begin{theorem}}
\newcommand{\ethm}{\end{theorem}}
\newcommand{\bprop}{\begin{proposition}}
\newcommand{\eprop}{\end{proposition}}
\newcommand{\bcor}{\begin{corollary}}
\newcommand{\ecor}{\end{corollary}}
\newcommand{\blem}{\begin{lemma}}
\newcommand{\elem}{\end{lemma}}
\newcommand{\bex}{\begin{example}\begin{rm}}
\newcommand{\eex}{\end{rm}\end{example}}

\newcommand{\N}{{\mathbb N}}

\newcommand{\ba}{\begin{array}} \newcommand{\ea}{\end{array}}

\newcommand{\eps}{\varepsilon}

\newcommand{\norm}[1]{\left\lVert#1\right\rVert}
\usepackage{mathtools}
\DeclarePairedDelimiter{\ceil}{\lceil}{\rceil}

\newcommand{\remin}{\mathop{-\!\!\!\!\!\hspace*{1mm}\raisebox{0.5mm}{$
\cdot$}}\nolimits}
\newcommand{\id}{{\rm id}}

\usepackage[left=2.0cm,%
                right=2.0cm,%
                top=2.5cm,%
                bottom=3.5cm,%
                headheight=12pt,%
                a4paper]{geometry}%

\begin{document}

\title[On computational properties of Cauchy problems]{On computational properties of Cauchy problems generated by accretive operators}

\author[Pedro Pinto and Nicholas Pischke]{Pedro Pinto and Nicholas Pischke}
\date{\today}
\maketitle
\vspace*{-5mm}
\begin{center}
{\scriptsize Department of Mathematics, Technische Universit\"at Darmstadt,\\
Schlossgartenstra\ss{}e 7, 64289 Darmstadt, Germany, \ \\ 
E-mail: $\{$pinto,pischke$\}$@mathematik.tu-darmstadt.de}
\end{center}

\maketitle
\begin{abstract}
In this paper, we provide quantitative versions of results on the asymptotic behavior of nonlinear semigroups generated by an accretive operator due to O. Nevanlinna and S. Reich as well as H.-K. Xu. These results themselves rely on a particular assumption on the underlying operator introduced by A. Pazy under the name of `convergence condition'. Based on logical techniques from `proof mining', a subdiscipline of mathematical logic, we derive various notions of a `convergence condition with modulus' which provide quantitative information on this condition in different ways. These techniques then also facilitate the extraction of quantitative information on the convergence results of Nevanlinna and Reich as well as Xu, in particular also in the form of rates of convergence which depend on these moduli for the convergence condition.
\end{abstract}

\section{Introduction}

One of the fundamental questions in the theory of differential equations is that of the asymptotic behavior of the solutions to a particular system. Concretely, consider the following initial value problem \cite{Bar1976,Paz1983} 
\[
\begin{cases}
u'(t)\in -Au(t),\, 0<t<\infty\\
u(0)=x
\end{cases}
\tag{$*$}
\]
over a Banach space $X$ generated by an initial value $x\in X$ and an accretive set-valued operator $A:X\to 2^X$ (see Section \ref{sec:prelim} for a precise definition). In that context, one calls a function $u:[0,\infty)\to X$ a solution of $(*)$ if $u(0)=x$, $u(t)$ is absolutely continuous, differentiable almost everywhere in $(0,\infty)$ and satisfies $(*)$ almost everywhere

It is straightforward to show that any solution is unique as $A$ is accretive (see e.g.\ \cite{Bar1976}). Unfortunately, the system is in general not solvable for $x\in\mathrm{dom}A$ even if $A$ is m-accretive as shown by Crandall and Liggett in \cite{CL1971b}. If the system is solvable, however, then one can consider the operator $S(t)x$ on $\mathrm{dom}A$ induced by the solutions $u_x(t)$ to $(*)$ with initial values $x\in\mathrm{dom}A$. This operator is continuous in $x$ and can thus be extended to $\overline{\mathrm{dom} A}$, thereby generating the semigroup $\mathcal{S}=\{S(t)\mid t\geq 0\}$ on $\overline{\mathrm{dom}A}$ associated with $(*)$ (see Section \ref{sec:prelim} for a precise definition of the notion of semigroup).

As studied in the fundamental paper of Brezis and Pazy \cite{BP1970}, these solutions have a particularly intriguing representation in terms of the so-called exponential formula:
\[
u_x(t)=\lim_{n\to\infty}\left(\mathrm{Id}+\frac{t}{n}A\right)^{-n}x.
\]
Extending the results of Brezis and Pazy, in \cite{CL1971b} Crandall and Liggett further showed that this formula in general always generates a nonexpansive semigroup. Concretely, a special case of their results yields that if $A$ is m-accretive, then the limit 
\[
S(t)x=\lim_{n\to\infty}\left(\mathrm{Id}+\frac{t}{n}A\right)^{-n}x
\]
exists for any $x\in\overline{\mathrm{dom}A}$ and any $t\geq 0$ and the $S(t)$ generated in that way form a nonexpansive semigroup on $\overline{\mathrm{dom}A}$. Note that by the above result of Brezis and Pazy, the semigroup thus generated generalizes the solution semigroup discussed above and Crandall and Liggett in \cite{CL1971b} even obtained a characterizing condition for when an $S(t)x$ as above actually represents a solution to $(*)$. Namely, their result yields in particular that if $0<T\leq \infty$ and $A$ is m-accretive, then $u_x$ is a solution of the initial value problem with $x\in\mathrm{dom}A$ on $[0,T)$ if and only if $u_x(t)=\lim_{n\to\infty}\left(\mathrm{Id}+\frac{t}{n}A\right)^{-n}x$ for $t\in [0,T)$ and $u_x$ is differentiable almost everywhere.

As the function $S(t)x$ induced by $\lim_{n\to\infty}\left(\mathrm{Id}+\frac{t}{n}A\right)^{-n}x$ is Lipschitz continuous in $t$ (see, e.g., the proof of Theorem 1.3 in \cite{Bar1976}, Chapter III), the additional differentiability condition is in particular immediately satisfied if any Lipschitz continuous function from the real numbers into $X$ is differentiable almost everywhere. This in turn is true in any reflexive space $X$ by (an extension of) Rademacher's theorem which, as is well-known, in particular includes uniformly convex spaces by the Milman–Pettis theorem.

In our paper, we are concerned with the asymptotic behavior of these semigroups of solutions for $t\to\infty$ in the context of uniformly convex and uniformly smooth spaces. It is well-known that $S(t)x$ does not always converge in that case. Motivated by these circumstances, there has been a search for potential conditions guaranteeing the convergence of the orbits of the semigroup generated by $A$ via the exponential formula. In that context, Pazy in \cite{Paz1978} introduced the so-called convergence condition for the operator $A$. Concretely, over a Hilbert space $X$ with inner product $\langle\cdot,\cdot\rangle$, we say (following Pazy) that an operator $A$ (assuming $A^{-1}0\neq\emptyset$) satisfies the convergence condition\footnote{Actually, Pazy also emphasized a particular consequence of the above condition as a separate additional property for the convergence condition, but we refrain from doing so (in line with the presentation in \cite{NR1979}).} if for all bounded sequences $(x_n,y_n)\subseteq A$ such that
\[
\lim_{n\to\infty}\langle y_n,x_n-Px_n\rangle=0,
\]
it holds that $\liminf_{n\to\infty}\norm{x_n-Px_n}=0$ where $P$ is the projection onto the closed and convex set $A^{-1}0$ (if $A$ is maximally monotone). Then Pazy obtained the following result:

\begin{theorem}[Pazy \cite{Paz1978}]
Let $X$ be a Hilbert space and $A$ be maximally monotone with $A^{-1}0\neq\emptyset$ and let $\mathcal{S}=\{S(t)\mid t\geq 0\}$ be the semigroup generated by $A$. If $A$ satisfies the convergence condition then, for every $x\in\overline{\mathrm{dom}A}$, $S(t)x$ converges strongly to a zero of $A$ as $t\to\infty$.
\end{theorem}

This convergence result was subsequently extended to uniformly convex and uniformly smooth Banach spaces by Nevanlinna and Reich in \cite{NR1979} who simultaneously adapted the above convergence condition to a suitable variant in said classes of Banach spaces by modifying the premise to the assumption that 
\[
\lim_{n\to\infty}\langle y_n,J(x_n-Px_n)\rangle=0
\]
where $J$ is the normalized-duality map. Concretely, the following result was obtained:

\begin{theorem}[Nevanlinna and Reich \cite{NR1979}]\label{thm:NR}
Let $X$ be uniformly convex and uniformly smooth and $A$ be $m$-accretive with $A^{-1}0\neq\emptyset$ and such that it satisfies the convergence condition. If $\mathcal{S}=\{S(t)\mid t\geq 0\}$ is the semigroup generated by $A$ via the exponential formula then, for any $x\in\overline{\mathrm{dom}A}$, $S(t)x$ converges strongly to a zero of $A$ as $t\to\infty$.
\end{theorem}

This result was further generalized by Xu in \cite{Xu2001} who studied the behavior of almost-orbits associated with the semigroup generated by $A$ as introduced by Miyadera and Kobayasi \cite{MK1982}: an almost-orbit of $\mathcal{S}$ is a continuous function $u:[0,\infty)\to\overline{\mathrm{dom}A}$ such that
\[
\lim_{s\to\infty}\sup\{\norm{u(t+s)-S(t)u(s)}\mid t\geq 0\}=0.
\]
Concretely, Xu obtained the following result:

\begin{theorem}[Xu \cite{Xu2001}]\label{thm:Xu}
Let $X$ be uniformly convex and uniformly smooth and $A$ be $m$-accretive with $A^{-1}0\neq\emptyset$ and such that it satisfies the convergence condition. If $\mathcal{S}=\{S(t)\mid t\geq 0\}$ is the semigroup generated by $A$ via the exponential formula, then every almost-orbit $u(t)$ of $\mathcal{S}$ converges strongly to a zero of $A$ as $t\to\infty$.
\end{theorem}

All the above results do not offer any quantitative information on the convergence of the orbits or almost-orbits. We resolve this in this paper by analyzing the proofs of Theorem \ref{thm:NR} as well as \ref{thm:Xu} and by extracting from that explicit computable transformations which translate a modulus witnessing a quantitative reformulation of the convergence condition into quantitative information on the convergence result. By this latter statement, we mean in particular full rates of convergence for $S(t)x$ as $t\to\infty$ in the context of the result of Nevanlinna and Reich. In the case of the result of Xu, this amounts to two kinds of quantitative ``translations'' with the first translating a rate of convergence for the almost-orbit into a rate of convergence of the solution of the Cauchy problem towards a zero of the operator $A$. Akin to fundamental results of Specker \cite{Spe1949} from recursion theory whereas even computable monotone sequences of rational numbers in $[0,1]$ do not have a computable rate of convergence, one can see that those rates will in general not be computable (see for similar results also the work of Neumann \cite{Neu2015}).\\

Even if computable rates of convergence are in general unattainable, one can, in very general situations, provide effective rates of so-called metastability which are, moreover, highly uniform. This notion of metastability has been recognized as an important finitary version of the Cauchy property from a non-logical perspective by Tao (see e.g.\ \cite{Tao2008b,Tao2008a}) and originates from a (noneffectively) equivalent but constructively weakened reformulation of the Cauchy property in some metric space $(X,d)$, say
\[
\forall k\in\mathbb{N}\ \exists n\in\mathbb{N}\ \forall i,j\geq n\left(d(x_i,x_j)< \frac{1}{k+1}\right)
\]
into 
\[
\forall k\in\mathbb{N}\ \forall g:\mathbb{N}\to\mathbb{N}\ \exists n\in\mathbb{N}\ \forall i,j\in [n;n+g(n)]\left(d(x_i,x_j)< \frac{1}{k+1}\right)
\]
where we write $[a;b]:=[a,b]\cap\mathbb{N}$, with a rate of metastability being a bound on `$\exists n\in\mathbb{N}$' in terms of $k$ and $g$. 

The second quantitative result on Theorem \ref{thm:Xu} then takes the form of a translation converting a rate of metastability of the almost-orbit (which will be discussed later on) into a rate of metastability for the convergence towards a zero of the operator $A$. For this, note in particular the example presented in \cite{KKA2015} for a concrete almost-orbit where such a rate of metastability can be naturally obtained and is moreover computable and highly uniform while any rate of convergence will not even be computable in this case.\\

Methodologically, the results of this work are based on the general approach and methods of the `proof mining' program, a discipline of mathematical logic which aims at the extraction of quantitative information from prima facie nonconstructive proofs by logical transformations (see \cite{Koh2008} for a book treatment and \cite{Koh2019} for a recent survey). In that vein, this work can in particular be viewed as a new case study in this program and is further strongly related to the only other previous foray of proof mining into the theory of partial differential equations and abstract Cauchy problems presented by Kohlenbach and Koutsoukou-Argyraki in \cite{KKA2015} (as well as to the only two other previous considerations on nonexpansive semigroups presented in \cite{KKA2016,KA2018}).\\

In particular, the theorem of Garc\'ia-Falset \cite{GF2005} analyzed by them is strongly related to the results of Pazy, Nevanlinna and Reich as well as Xu presented above. Concretely, Garc\'ia-Falset obtains a similar result on the asymptotic behavior of the almost-orbits of the solution semigroup of the abstract Cauchy problem generated by an operator $A$ under the condition that $A$ is $\phi$-accretive at zero as defined in \cite{GF2005}. The generality gained by assuming $\phi$-accretivity at zero of $A$ is that the space is allowed to be an arbitrary Banach space.

In that context, our dichotomous situation of the two quantitative versions of the result of Xu is also similar to the results from \cite{KKA2015} and, as will be discussed later, the work \cite{KKA2015} is where the metastable version of the almost-orbit condition was first introduced.\\

In contrast to the results by Garc\'ia-Falset in \cite{GF2005} where the notion of $\phi$-accretive at zero carries the strength of removing the convergence condition as well as the assumptions on the space $X$ but simultaneously provides a strong restriction on the operator (by, among others, making the zero of the operator unique), the results given by Pazy, Nevanlinna and Reich as well as Xu offer a practically higher generality at the modest price of a uniformly convex and uniformly smooth space, a property which is still fulfilled for most spaces of interest, in particular for all $L^p$-spaces as is the case for all examples of application given in \cite{GF2005}.\\

Further, these assumptions on the space and the operator offer a rich complexity of interworking notions, all having logically interesting properties which in particular crucially rest on many of the recent logical insights into accretive operators obtained in \cite{Pis2022}, making this case study especially interesting in the context of the proof mining program. In particular, the logically motivated quantitative considerations on the convergence condition are independently of interest beyond the scope of this paper as they, for one, offer a new perspective on the notion of being $\phi$-accretive at zero and its quantitative version introduced by Kohlenbach and Koutsoukou-Argyraki in \cite{KKA2015} (as will be discussed later on), and, for another, as the convergence condition features in many other results on the approximation of zeros of accretive operators like, e.g., the iteration schemes considered by Nevanlinna and Reich in \cite{NR1979} or the asymptotic behavior of incomplete Cauchy problems as considered by Poffald and Reich \cite{PR1986}. In that way, future quantitative analyses of such results will depend on these moduli for the convergence condition. Lastly, we further find that a logical analysis of the results of Nevanlinna and Reich as well as Xu also yields a qualitative improvement on said theorems. Concretely, the analysis reveals exactly the assumptions necessary on the duality map and projection involved in the proofs and we in that way obtain a generalization of the results to arbitrary Banach spaces which only satisfy some weak additional requirements on said mappings.\\

As an outline of the paper, we begin by discussing some preliminaries in Section \ref{sec:prelim}. Section \ref{sec:CC} is then devoted to the study of the convergence condition and its quantitative versions as mentioned previously which are crucially used in Section \ref{sec:quantRes} where we present quantitative versions of the previously discussed results contained in Theorems \ref{thm:NR} and \ref{thm:Xu}. The logical properties of these results and their extractions in the context of the proof mining program will be discussed in Section \ref{sec:Logic}.

\section{Preliminaries}\label{sec:prelim}

\subsection{Convexity and smoothness in Banach spaces}
Consider a Banach space $(X,\norm{\cdot})$. We use $X^*$ to denote the dual of $X$, i.e.\ the set of all linear continuous functionals $x^*:X\to\mathbb{R}$. We assume throughout that $X$ is \emph{uniform convex}, i.e.
\[
\forall \eps\in(0, 2]\ \exists \delta \in (0,1]\ \forall x, y\in B_1(0) \left( \|x-y\|\geq \eps \to \left\|\frac{x+y}{2}\right\|\leq 1-\delta \right),
\]
and \emph{uniformly smooth}, i.e.
\[
\forall \eps > 0\ \exists \delta >0\ \forall x, y\in X \left( \|x\|=1 \land \|y\|\leq \delta \to \|x+y\|+ \|x-y\|\leq 2 + \eps\|y\| \right).
\] 
Note that $X$ is uniformly convex iff its dual $X^*$ is uniformly smooth.\\

Associated with $X$ is the \emph{normalized duality mapping} $J:X \to 2^{X^*}$, defined by
\[
J(x):=\{ f\in X^*\mid f(x)=\|x\|^2\text{ and }\|f\|^*=\|x\| \},
\]
for all $x\in X$ where we write $\norm{\cdot}^*$ for the dual norm on $X^*$. This mapping is single-valued and uniformly continuous if, and only if, $X$ is uniformly smooth (see \cite{Chi2009}). As is common in that context, we identify $J$ with this unique mapping $X\to X^*$ and write $\langle y, J(x)\rangle$ instead of $J(x)(y)$ given $x,y\in X$.\\

As $X$ is uniformly convex, if $C\subseteq X$ is a nonempty, closed, convex subset of $X$, then the \emph{nearest point projection} $P_C:X\to C$ is single-valued and outputs the unique point satisfying the condition
\[
\|x-P_Cx\| = \inf\{ \|x-y\|\mid y\in C \}.
\]
Even further, the projection map $P_C$ is continuous and even uniformly continuous in uniformly convex spaces as will be used later (see \cite{Sch1971}).

\subsection{Accretive operators}

Throughout, we will be concerned with set-valued operators $A:X\to 2^X$ on the space $X$. As such an operator is set-theoretically nothing else than its graph, i.e.\ a relation $A\subseteq X\times X$, we use both notations $(x,y)\in A$ and $y\in Ax$ interchangeably. The operator $A$ is called \emph{accretive} if
\[
\forall (x_1, y_1), (x_2, y_2)\in A \left( \langle y_1-y_2, J(x_1 - x_2)\rangle \geq 0 \right).\label{p:accretive}
\]
These operators originate in the work of Kato \cite{Kat1967} (based itself on works of Minty, Browder and others). In fact, the connection between these operators and partial differential equations (as studied here) has already been a motivating consideration since Kato's groundbreaking work. For further background on these connections we refer to the standard reference of Barbu \cite{Bar1976} and for further background on accretive operators, we also refer to the classical text of Takahashi \cite{Tak2000}.

We say that $A$ is \emph{m-accretive} if $\mathrm{ran}(Id+\gamma A)=X$ for all $\gamma >0$ in addition to being accretive. We write $\mathrm{dom}A:=\{ x\in X \mid Ax \neq \emptyset \}$ for the domain and $\mathrm{ran}A:=\bigcup_{x\in X}Ax$ for the range of $A$.\\

One of the main tools for studying these classes of set-valued operators is their \emph{resolvent} $J^A_\gamma$, which is defined as 
\[
J^A_\gamma:=(Id+\gamma A)^{-1}
\]
for $\gamma>0$. Immediately, we see that $J^A_\gamma$ (as a set-valued mapping) satisfies $\mathrm{dom}J^A_\gamma=\mathrm{ran}(Id+\gamma A)$ and $J^A_\gamma x\subseteq\mathrm{dom}A$ for all $x$. Fundamental to the resolvent in the context of accretive operators are the following characterizing properties: $A$ is accretive if, and only if $J^A_\gamma$ is single-valued and nonexpansive for every $\gamma>0$, or alternatively if, and only if $J^A_\gamma$ is single-valued and firmly nonexpansive in the sense of \cite{BR1977} for all/some $\gamma>0$. Proofs for these equivalences can be found in the standard references mentioned above. Further, by definition $J^A_\gamma$ is total for every $\gamma>0$ if $A$ is $m$-accretive.\\

Rather immediately one can see that the set $A^{-1}0$ coincides with the set of fixed points of the resolvent functions. If $A$ is m-accretive, it is also maximally accretive. Consequently, note that if $A$ is m-accretive, then $A^{-1}0$ is closed and convex. We henceforth write $P$ for $P_{A^{-1}0}$ which is well-defined in uniformly convex spaces if $A^{-1}0\neq\emptyset$. Note though, that in contrast to monotone operators on Hilbert space and Minty's theorem \cite{Min1962}, the converse does not hold as first asked in \cite{CP1969} and then answered in \cite{Cal1970,CL1971} negatively. 

\subsection{Nonlinear semigroups and the exponential formula}

As discussed in the introduction, the main structure of concern in this paper is that of a nonexpansive semigroup\footnote{The name derives from the fact that $\mathcal{S}$ becomes a semigroup if the group operation is taken to be the composition of the functions $S(t):C\to C$ which is well-defined as item (ii) guarantees that the set $\mathcal{S}$ is closed under this operation (which is naturally associative).} over a set $C$ by which we  will concretely mean a set $\mathcal{S}=\{S(t)\mid t\geq 0\}$ of functions $S(t):C\to C$ such that {\label{p:semigroups}}
\begin{enumerate}
\item $S(0)x=x$ for $x\in C$,
\item $S(t+s)x=S(t)S(s)x$ for $t,s\geq 0$ and $x\in C$,
\item $S(t)x$ is continuous in $t\geq 0$ for $x\in C$,
\item For any $t\geq 0$ and any $x,y\in C$: $\norm{S(t)x-S(t)y}\leq\norm{x-y}$. 
\end{enumerate}

As discussed in the introduction as well, these semigroups arise naturally in the context of initial value problems associated with an accretive operator and in that context can be characterized by the exponential formula as shown by Crandall and Liggett \cite{CL1971b}. Precisely, the following result was established therein:

\begin{theorem}[Crandall and Liggett \cite{CL1971b}]
Let $X$ be a Banach space and $A$ an m-accretive operator. Then
\[
S(t)x:=\lim_{n\to\infty}\left(\mathrm{Id}+\frac{t}{n}A\right)^{-n}x
\]
exists for all $x\in\overline{\mathrm{dom}A}$ and $t\geq 0$ and $\mathcal{S}=\{S(t)\mid t\geq 0\}$ is a nonlinear semigroup on $\overline{\mathrm{dom}A}$. Further, if $X$ is reflexive, then for $x\in\mathrm{dom}A$, $S(t)x$ is a solution for the system \emph{($*$)} at $x$.
\end{theorem}

In the following, we call $\mathcal{S}$ the semigroup generated by $A$ via the exponential formula.

\section{The convergence condition and quantitative versions}\label{sec:CC}

As discussed in the introduction, the central notion for the asymptotic results from \cite{NR1979,Paz1978,Xu2001} is that of the convergence condition for the operator $A$ inducing the differential equation. In the quantitative versions of these results of Pazy, Reich and Nevanlinna as well as Xu, we will rely on a (or rather multiple) particular quantitative version(s) of that condition, which we shall call a convergence condition with a modulus. These quantitative reformulations are motivated by logical considerations on different equivalent variants of the convergence condition in the spirit of the proof mining program (see in particular Section \ref{sec:Logic} later on for a discussion of these logical aspects). In particular, we will discuss in Section \ref{sec:Logic} that these moduli have the two central properties that general logical metatheorems underlying the whole logical approach of proof mining guarantee, for one, the extractability of such moduli for a large class of operators which provably satisfy the convergence condition and, for another, the same logical metatheorems guarantee that from any proof confined by certain general logical conditions and using the assumption that an operator satisfies the convergence condition, quantitative results can be extracted which depend on such a modulus.

\subsection{Variants of the convergence condition}

To begin with, as mentioned in the introduction, the original formulation of the convergence condition is due to Pazy \cite{Paz1978}, but in our setting of uniformly convex and uniformly smooth Banach spaces, we follow the notion of Nevanlinna and Reich \cite{NR1979} and, therefore, say that an $A$ with $A^{-1}0\neq\emptyset$ satisfies the \emph{convergence condition} if for all bounded sequences $(x_n,y_n) \subseteq A$:
\[
\lim \langle y_n, J(x_n-Px_n)\rangle = 0\ \to\ \liminf \|x_n-Px_n\|=0.
\]
Already in the literature, other equivalent variants are sometimes mentioned, e.g.\ replacing the limit in the premise of the implication by a limit inferior or conversely replacing the limit inferior in the conclusion by a limit (see for example \cite{PR1986}). However, in the following we only focus on the usual formulation of the convergence condition in the form above, together with one particular equivalent version which is of a different spirit entirely:
\begin{lemma}\label{lem:CCequiv}
An operator $A$ satisfies the convergence condition if, and only if, for all natural numbers $k, K\in\N$, there exists $n\in\N$ such that\footnote{The absolute values are actually not necessary in the premise as $A$ is accretive.}
\[
\forall (x,y)\in A\left(\|x\|, \|y\|\leq K \land \left\vert\langle y, J(x-Px)\rangle\right\vert \leq \frac{1}{n+1} \to \|x-Px\|\leq \frac{1}{k+1} \right).\tag{$+$}
\]
\end{lemma}
\begin{proof}
For sufficiency, consider arbitrary sequences $(x_n), (y_n)$ such that $y_n\in Ax_n$, and $\|x_n\|, \|y_n\|\leq K$ for some $K\in\N$. Assume that $\lim \langle y_n, J(x_n-Px_n)\rangle = 0$ and let $k\in\N$ be given. By $(+)$, there is an $n\in\N$ such that
\[
\forall m\in\N \left( \left\vert\langle y_m, J(x_m-Px_m)\rangle\right\vert \leq \frac{1}{n+1} \to \|x_m-Px_m\|\leq \frac{1}{k+1} \right).\tag{$++$}
\]
Then, by $\lim\langle y_n,J(x_n-Px_n)\rangle=0$ there exists $N\in\mathbb{N}$ such that
\[
\forall m\geq N \left(\left\vert\langle y_m, J(x_m-Px_m)\rangle\right\vert \leq \frac{1}{n+1}\right),
\]
which by $(++)$ entails that $\|x_m-Px_m\|\leq \frac{1}{k+1}$, for all $m\geq N$. This means that $\lim \|x_n-Px_n\|=0$, and we conclude that $A$ satisfies the convergence condition.\\

\noindent For necessity, suppose that $(+)$ fails. Then for some $k, K\in\N$, we have
\[
\begin{split}
&\forall n\in\N \ \exists (x_n,y_n)\in A\\
&\left( \|x_n\|, \|y_n\|\leq K \land \left\vert\langle y_n, J(x_n-Px_n)\rangle\right\vert \leq \frac{1}{n+1} \land \|x_n-Px_n\| > \frac{1}{k+1} \right).
\end{split}
\]
Then in particular $\left\vert\langle y_n, J(x_n-Px_n)\rangle\right\vert \leq \frac{1}{n+1}$ for all $n\in\mathbb{N}$ which entails that
\[
\lim \langle y_n, J(x_n-Px_n)\rangle =0.
\]
However $(\|x_n-Px_n\|)$ is bounded away from zero by $\frac{1}{k+1}$, and so $A$ can not satisfy the convergence condition.
\end{proof}

The above equivalent version does not feature sequences at all and, in this way, is of a much more local nature than the original formulation. By applying the underlying logical considerations of proof mining to these two formulations, we will now derive the previously mentioned quantitative versions of the convergence condition in the form of two different moduli (where this difference of the moduli can actually be recognized in terms of logical properties of their equivalence proof as will be discussed in Section \ref{sec:Logic} later on). We want to note that both the above equivalence and the following quantitative versions are similar in character to the alternative characterization of strongly nonexpansive mappings introduced in \cite{Koh2016} as well as the moduli introduced there.

\subsection{Quantitative versions of the convergence condition}

Note that the convergence condition is essentially (modulo the boundedness condition) of the general form
\[
\lim a_n=0\rightarrow \liminf b_n=0
\]
with $a_n=\langle y_n,J(x_n-Px_n)\rangle$ and $b_n=\norm{x_n-Px_n}$. In that conceptual vein, two of our quantitative versions of the convergence condition will be certain \emph{moduli} translating a quantitative witness for the convergence $\lim a_n=0$ in the premise into a quantitative witness for $\liminf b_n=0$ in the conclusion (or even for a weakening of that).\\

In that way, two of these moduli arise by considering combinations of a quantitative witness for the convergences in the premise or conclusion and for that, we rely on the following notions providing such a quantitative account in various ways:
\begin{definition}
Let $(a_n)$ be a sequence of non-negative real numbers. 
\begin{enumerate}
\item We say that a functional $\varphi:\mathbb{N}\to\mathbb{N}$ is a rate of convergence for $(a_n)$ (towards zero) if
\[
\forall k\in\mathbb{N}\ \forall n\geq \varphi(k)\left(a_n\leq\frac{1}{k+1}\right).
\]
\item We say that a functional $\varphi:\N\times\N \to \N$ is a $\liminf$-rate for $(a_n)$ (towards zero) if
\[
\forall k, m\in\mathbb{N}\ \exists n\in [m;\varphi(k,m)] \left(a_n\leq \frac{1}{k+1} \right).
\]
\item We say that a functional $\varphi:\mathbb{N}\to\mathbb{N}$ is a rate of approximate zeros for $(a_n)$ if
\[
\forall k\in\mathbb{N}\ \exists n\leq\varphi(k)\left(a_n\leq\frac{1}{k+1}\right).
\]
\end{enumerate}
\end{definition}

Combinations of these quantitative versions of $\lim/\liminf =0$ (or even of the weaker property of approximate zeros) now yield the previously mentioned different quantitative versions of the convergence condition. We begin with the most immediate version which translates a rate of convergence for the premise together with the upper bound on the sequence into a $\liminf$-rate for the conclusion.

\begin{definition}\label{def:CCmod}
A \emph{modulus for the convergence condition} is a functional $\Omega:\N\times \N^{\N} \to \N^{\N\times\N}$ such that for any $(x_n), (y_n) \subseteq X$ and any $K\in\N$ and $\varphi:\mathbb{N}\to\mathbb{N}$:
\begin{align*}
&\textbf{if }\forall n\in\mathbb{N}\left(y_n \in Ax_n\land \|x_n\|,\|y_n\|\leq K\right)\\
&\text{and }\varphi\text{ is a rate of convergence for }\, \vert \langle y_n,J(x_n-Px_n)\rangle\vert,\\
&\textbf{then }\Omega(K,\varphi)\text{ is a $\liminf$-rate for }\norm{x_n-Px_n}.
\end{align*}
\end{definition}

While conceptually appealing due to its naturality, the logical considerations underlying the approach of proof mining actually in general suggest a stronger type of modulus, named a \emph{full} modulus here, to be necessary in the context of general quantitative analyses of results relying on the convergence condition as well as classical logic. Actually, we will present two general logical metatheorems in Section \ref{sec:Logic} that guarantee both
\begin{enumerate}\label{enum:items}
\item the extractability of a computable full modulus (and thus of a `simple' modulus) for the convergence condition from a wide range of (noneffective) proofs of the convergence condition for definable classes of operators, as well as,
\item that from a proof using the convergence condition as a premise, a transformation can be extracted that transforms 
\begin{enumerate}
\item a full modulus into quantitative information on the conclusion if the underlying proof is nonconstructive,
\item a `simple' modulus into quantitative information on the conclusion if the underlying proof is `essentially' constructive,
\end{enumerate}
where, moreover, the complexity of the principles used in the proof is reflected in the complexity of the extracted transformation.
\end{enumerate}
In that way, while the above modulus is derived from a `constructive' perspective on the convergence condition, the following full modulus is attained from a `classical' perspective on it. We however postpone a detailed discussion of these logical aspects to the end of the paper (see Section \ref{sec:Logic}) where in particular we will give formal justifications for the above statements. For the rest of this section, and for the rest of the paper, up to Section \ref{sec:Logic} for that matter, we just state the full modulus as another notion which provides a quantitative perspective on the convergence condition and discuss examples of such a modulus for various classes of operators.

\begin{definition}\label{def:CCmodFull}
A \emph{full modulus for the convergence condition} is a functional $\Omega^f:\N\times \N \to \N$ satisfying that for any $k, K\in\N$: if $y\in Ax$ are such that $\|x\|, \|y\|\leq K$, then 
\[
\vert\langle y,J(x-Px)\rangle\vert\leq \frac{1}{\Omega^f(K,k)+1}\Rightarrow\norm{x-Px}\leq \frac{1}{k+1}.
\]
\end{definition}

In that way, the above modulus is not of the general pattern laid out before that converts quantitative information on the limit in the premise into quantitative information on the limit in the conclusion but rather represents a kind of local perspective that already transfers local errors of the conclusion into local errors for the premise. In a way, the above is a true finitization of the convergence condition in the sense that the above notion only refers to finitely many objects together with the fact that by the result given in Lemma \ref{lem:CCequiv}, we have effectively shown the following:
\begin{proposition}
An operator $A$ satisfies the convergence condition if, and only if, it has a full modulus for the convergence condition $\Omega^f$.
\end{proposition}
\begin{remark}\label{rem:strictSNEremark}
Note by Lemma \ref{lem:CCequiv} that the convergence condition is nothing else but a uniform version of the property
\[
\forall (x,y)\in A\ \forall k\in\mathbb{N}\ \exists n\in\mathbb{N} \left( \vert \langle y, J(x-Px)\rangle\vert \leq\frac{1}{n+1} \to \|x-Px\|< \frac{1}{k+1} \right),
\]
which can easily be seen to be equivalent to
\[
\forall (x,y)\in A\, \left( \langle y, J(x-Px)\rangle =0 \to \|x-Px\|=0 \right).
\]
This property was already singled out as an important special case of the convergence condition in Pazy's original paper \cite{Paz1978} (as mentioned already in a footnote in the introduction). In particular, based on the logical form of the above statement, the logical metatheorems mentioned above actually guarantee a strengthened form of item (i) in the sense that already from a (possibly noneffective) proof of the above property for a class of operators, one can extract a computable full modulus (and thus a `simple' modulus) for the convergence condition, provided the proof is as before confined by the logical conditions of the metatheorem. Also this situation is conceptually similar to the results on strongly nonexpansive mappings from \cite{Koh2016}, in particular to the fact that the SNE-modulus introduced there arises as the uniform version of the notion of strict nonexpansivity.
\end{remark}

In any way, even in the case of a (semi-)constructive proof and in the context of a `simple' modulus, the required modulus can often further be weakened. While our quantitative versions of the convergence results of Nevanlinna and Reich as well as Xu can, for one, be stated already in terms of a `simple' modulus for the convergence condition, the only sequences to which the convergence condition is ever applied (in the context of this paper) are such that $\norm{x_n-Px_n}$ is nonincreasing. In that case, it is clear that it already suffices to require a modulus which translates a rate of convergence $\varphi$ for the sequence $\vert \langle y_n,J(x_n-Px_n)\rangle\vert$ together with the bound $K$ into a rate of approximate zeros $\Omega(K,\varphi)$ for the sequence $\norm{x_n-Px_n}$. As this circumstance seems to occur rather frequently,\footnote{In fact, in e.g.\ the related work \cite{GF2005} on quantitative behavior of semigroups generated by $\phi$-accretive operators, the requirements in the condition of $\phi$-accretivity (essentially replacing the convergence condition) are such that they restrict the conclusion essentially to sequences $x_n$ such that $\norm{x_n-Px_n}$ is decreasing. A similar restriction could have been made in the case of the convergence condition since, as said above, the applications given in \cite{NR1979,Paz1978,Xu2001} satisfy the requirement but it seems that the authors have refrained from doing so to make the condition less technical.} we introduce this special case as a particular other notion for a quantitative form of the convergence condition:
\begin{definition}
A \emph{weak modulus for the convergence condition} is a functional $\Omega^w:\N\times \N^{\N} \to \N^{\N}$ such that for any $(x_n), (y_n) \subseteq X$ and any $K\in\N$ and $\varphi:\mathbb{N}\to\mathbb{N}$:
\begin{align*}
&\textbf{if }\forall n\in\mathbb{N}\left(y_n \in Ax_n\land \|x_n\|,\|y_n\|\leq K\right)\\
&\text{and }\varphi\text{ is a rate of convergence for }\, \vert \langle y_n,J(x_n-Px_n)\rangle\vert,\\
&\textbf{then }\Omega^w(K,\varphi)\text{ is a rate of approximate zeros for }\norm{x_n-Px_n}.
\end{align*}
\end{definition}
In that way, while both the full and `simple' moduli represent the correct quantitative content of the convergence condition (from a classical and a constructive perspective, i.e.\ complying with the properties (i) and (ii) mentioned above, respectively), the extractions formulated here will be phrased in terms of the weaker quantitative assumption of a weak modulus for the convergence condition. Note for this that there is of course no loss of generality as given a full modulus $\Omega^f$, a `simple' modulus $\Omega$ can be defined via $\Omega(K,\varphi)(k,m)=\max\{m,\varphi(\Omega^f(K,k))\}$ and in turn, given a `simple' modulus $\Omega$, a weak modulus $\Omega^w$ can be defined just via $\Omega^w(K,\varphi)(k)=\Omega(K,\varphi)(k,0)$.

\subsection{Examples for operators and their moduli}

In the following, we survey various examples given in the works \cite{NR1979,Paz1978} and beyond for classes of operators which naturally satisfy the convergence condition. Based on the corresponding proofs, we extract respective full moduli in the sense of the previous section.

\subsubsection{Strongly accretive operators}
The following is an immediate generalization of Example 4.3 in \cite{Paz1978}.
\begin{lemma}
If $A$ is \emph{strongly accretive}, by which we mean there exists an $\alpha>0$ such that
\[
\langle u-v,J(x-y)\rangle\geq\alpha\norm{x-y}^2
\]
for any $(x,u),(y,v)\in A$ and additionally $A^{-1}0\neq\emptyset$, then $A$ satisfies the convergence condition with a full modulus for the convergence condition $\Omega^f_a(K,k)=a(k+1)^2\remin 1$ for any $a\in\mathbb{N}^*$\footnote{We write $\mathbb{N}^*$ for $\mathbb{N}\setminus\{0\}$.} such that $\alpha\geq a^{-1}$.
\end{lemma}
\begin{proof}
Let $(x,y)\subseteq A$ with $\norm{x},\norm{y}\leq K$ and where 
\[
\langle y,J(x-Px)\rangle\leq\frac{1}{\Omega^f_a(K,k)+1}.
\]
Then as $\langle y,J(x-Px)\rangle\geq\alpha\norm{x-Px}^2$ we get
\[
\alpha\norm{x-Px}^2\leq \frac{1}{\Omega^f_a(K,k)+1}\leq \frac{1}{a(k+1)^2}\leq\frac{\alpha}{(k+1)^2}
\]
which yields $\norm{x-Px}\leq 1/(k+1)$.
\end{proof}
As already mentioned in \cite{Paz1978}, a particular example of a strongly monotone operator is the negative Laplacian: Let $\Omega$ be a bounded domain in $\mathbb{R}^n$ with smooth boundary. $\mathrm{L}^2(\Omega)$ is the space of square-integrable functions as usual and $\mathrm{W}^{1,2}_0(\Omega)$ the associated subspace of the Sobolev-space $\mathrm{W}^{1,2}(\Omega)$ containing functions of zero-trace. Then using Poncair\'e's inequality (see e.g.\ \cite{Leo2009}), we get that
\[
-\int_\Omega\Delta u\cdot u\,\mathrm{dx}=\int_\Omega\vert\nabla u\vert^2\,\mathrm{dx}\geq\lambda_1\int_\Omega\vert u\vert^2\,\mathrm{dx}
\]
where $\Delta$ is the usual Laplacian operator and $\lambda_1>0$ is the minimal eigenvalue of $-\Delta$. Therefore, $A=-\Delta$ is strongly monotone and by the above lemma satisfies the convergence condition with a full modulus for the convergence condition
\[
\Omega^f_{\Lambda}(K,k)=\Lambda(k+1)^2\remin 1,
\]
where $\Lambda\in\mathbb{N}^*$ is such that $\Lambda^{-1}$ is a lower bound on the eigenvalues of $-\Delta$.

\subsubsection{Operators that are $\phi$-accretive at zero or uniformly accretive at zero}
The above case of strongly monotone operators is a special case of the notion of operators which are $\phi$-accretive at zero introduced in \cite{GF2005} over general Banach spaces.
\begin{definition}[\cite{GF2005}]
An operator $A$ with $0\in Az$ is $\phi$-accretive at zero in the sense of \cite{GF2005} if $\phi:X\to [0,\infty)$ is a continuous function with $\phi(0)=0$, $\phi(x)>0$ for $x\neq 0$ and
\[
\phi(x_n)\to 0\Rightarrow \norm{x_n}\to 0
\]
for every sequence $(x_n)\subseteq X$ such that $\norm{x_n}$ is nonincreasing and we have that
\[
\langle y, J(x-z)\rangle\geq\phi(x-z)
\]
for all $(x,y)\in A$. 
\end{definition}
As already mentioned in \cite{GF2005}, it is a straightforward consequence of \cite[Theorem 8]{GM2005} that if $A$ is m-$\psi$-strongly accretive in the sense of \cite{GF2005}, then $A$ is $(\psi\circ\norm{\cdot})$-accretive at zero.

In the course of their proof-theoretic analysis of the main result of \cite{GF2005}, which is similar in kind to the results analyzed here, Kohlenbach and Koutsoukou-Argyraki in \cite{KKA2015} introduced (similarly motivated by proof-theoretic considerations) a generalized uniform version of the above property (without any reference to a function $\phi$) under the name of uniform accretivity at zero:
\begin{definition}[\cite{KKA2015}]\label{def:unifAccAtZ}
An accretive operator $A$ with $0\in Az$ is called \emph{uniformly accretive at zero} if for all $k\in\mathbb{N}$ and all $K\in\mathbb{N}^*$, there exists an $m\in\mathbb{N}$ such that
\[
\forall (x,u)\in A\,\left(\norm{x-z}\in [2^{-k},K]\rightarrow \langle u,x-z\rangle_+\geq 2^{-m}\right)
\]
with $\langle\cdot,\cdot\rangle_+$ defined by
\[
\langle y,x\rangle_+:=\max\{\langle y,j\rangle\mid j\in J(x)\}.
\]
\end{definition}
This notion was accompanied in \cite{KKA2015} with a corresponding uniform quantitative modulus of being uniformly accretive at zero which is defined in the following sense:
\begin{definition}[\cite{KKA2015}]
A function $\Theta:\mathbb{N}\times\mathbb{N}^*\to\mathbb{N}$ is a \emph{modulus of accretivity at zero for $A$} if $m:=\Theta_K(k)$ satisfies the condition in Definition \ref{def:unifAccAtZ}.
\end{definition}
Note that this notion in particular encompasses the moduli of uniform $\phi$-accretivity at zero also introduced in \cite{KKA2015} which provide a quantitative perspective on the above notion of $\phi$-accretivity at zero.\\

Now, while our setting is more restrictive in terms of the space, we can nevertheless recognize the above notion as essentially stating the existence a full modulus for the convergence condition for $A$, at least in our context of uniformly convex and uniformly smooth spaces: At first, the expression $\langle u,x-z\rangle_+$ reduces to $\langle u,J(x-z)\rangle$ in a uniformly smooth space while in the context of uniformly convex spaces, through the presence of the projection $P$ and as the zero $z$ is unique, the point $z$ can be replaced by the projection $Px$ for any point $x$. Reading the resulting condition as its contraposition, we obtain that a modulus of accretivity at zero for $A$ satisfies that for any $k$ and $K$, if $\norm{x-Px}\leq K$, then
\[
\forall (x,u)\in A\,\left(\left\vert\langle u,J(x-Px)\rangle\right\vert < 2^{-\Theta_K(k)}\rightarrow \norm{x-Px}< 2^{-k}\right).
\]
Since we can bound $\norm{x-Px}$ by
\[
\norm{x-Px}\leq\norm{x}+\norm{z}
\]
using the single witness $z\in\mathrm{zer}A$ for $\mathrm{zer}A\neq\emptyset$ (as required in the context of the convergence condition), we get that therefore $\Omega^f$ defined by
\[
\Omega^f(K,k)=2^{\Theta_{K+Z}(k)}
\]
where $Z\geq\norm{z}$, is a full modulus for the convergence condition of $A$ which is even independent of an upper bound for $u\in Ax$. In that way, we find that the notion of being uniformly accretive at zero is essentially an equivalent formulation of the convergence condition in that context, which was moreover discovered by Kohlenbach and Koutsoukou-Argyraki by applying the same logical methodology as is underlying this work.

Thus, if restricted to the class of spaces considered here, we find that the quantitative results on the behavior of the semigroups generated by $A$ as derived in \cite{KKA2015} can also be recognized as applications of our general quantitative results, using the notion of a full modulus for the convergence condition $\Omega^f$.

\subsubsection{Operators without unique zeros}

All operators discussed so far are $\phi$-accretive in the sense of \cite{GF2005}. The convergence condition however encompasses a far larger class of operators and the difference set of those two notions is already populated with fairly simple examples of which we exhibit one in the following. For this, we recall the following result due to Pazy:

\begin{proposition}[Pazy \cite{Paz1978}]\label{pro:PazyProp}
Let $\varphi:X\to\overline{\mathbb{R}}$ be proper, convex and l.s.c. on a Hilbert space $X$ and assume that $\varphi(x)\geq 0$ for all $x\in X$ as well as $\min_{x\in X}\varphi(x)=0$. If the level-sets
\[
K_R=\{x\mid\norm{x}\leq R,\varphi(x)\leq R\}
\]
are totally bounded, then the maximally monotone operator \[
\partial\varphi(x)=\{u\in X\mid \varphi(y)-\varphi(x)\geq \langle y-x,u\rangle\text{ for any }y\in X\}
\]
satisfies the convergence condition.
\end{proposition}

Now, for an example of an operator which satisfies the convergence condition but is not $\phi$-accretive at zero for any $\phi$, consider the following function $f:\mathbb{R}\to\mathbb{R}$:
\[
f(x)=\begin{cases}
(x+1)^4&\text{if }x\in (-\infty,-1],\\
0&\text{if }x\in [-1,1],\\
(x-1)^4&\text{if }x\in [1,\infty).
\end{cases}
\]
This function is continuously differentiable with first derivative
\[
f'(x)=\begin{cases}
4(x+1)^3&\text{if }x\in (-\infty,-1],\\
0&\text{if }x\in [-1,1],\\
4(x-1)^3&\text{if }x\in [1,\infty).
\end{cases}
\]
Therefore $\partial f(x)=\{f'(x)\}$ for any $x\in\mathbb{R}$ and it is easy to see that $f$ is convex and that the level sets $K_R$ are compact. Thus $\partial f$ satisfies the convergence condition. However, we have $\mathrm{zer}\partial f=[-1,1]$ and thus $\partial f$ does not have a unique zero. The uniqueness of the zero is, however, a property of every operator that is $\phi$-accretive at zero (see \cite{GF2005}) and even of every operator which is uniformly accretive at zero (see \cite{KKA2015}).

Nevertheless, by a quantitative analysis of the application of Proposition \ref{pro:PazyProp} to the function $f$, we can immediately extract a full modulus for the convergence condition $\Omega^f(K,k)=(k+1)^4- 1$ for the convergence condition of $\partial f$: Let consider $x\in\mathbb{R}$ and assume $\vert x\vert,\vert f'(x)\vert\leq K$ as well as
\[
\langle y,x-Px\rangle\leq\frac{1}{((k+1)^4- 1)+1}.
\]
As in \cite{Paz1978}, one can show $\langle y,x-Px\rangle\geq f(x)$. Thus in particular
\[
f(x)\leq\frac{1}{((k+1)^4- 1)+1}.
\]
Generically, one can immediately show that if $f(x)\leq\varepsilon$ for $\varepsilon>0$, then $x\in [-1-\sqrt[4]{\varepsilon},1+\sqrt[4]{\varepsilon}]$ and thus $\norm{x-Px}\leq \sqrt[4]{\varepsilon}$. Therefore the above implies
\[
\norm{x-Px}\leq\frac{1}{k+1}
\]
as desired.

\section{Quantitative results on the asymptotic behavior of semigroups and their almost-orbits}\label{sec:quantRes}

In this section, we employ the previous quantitative considerations on the convergence condition for establishing quantitative versions of the theorems of Nevanlinna and Reich as well as of Xu outlined in the introduction. Note that since the proofs of the respective results are essentially constructive, a dependence on a `simple' (or even weak) modulus for the convergence condition can be guaranteed a priori for the extracted results (see the logical remarks in Section \ref{sec:Logic}) which is also the case for the concrete rates presented below. In that vein, we in the following denote all moduli just by an $\Omega$ without the previous superscripts. We begin with the result of Nevanlinna and Reich.

\subsection{The asymptotic behavior of nonlinear semigroups}

Consider again the setup from Theorem \ref{thm:NR} and write $\mathcal{S}=\{S(t)\mid t\geq 0\}$ for the semigroup generated by $A$ via the exponential formula. In following, if not stated otherwise, let $x\in\mathrm{dom}A$. We write $w_x(t)$ for $S(t)x$ (in the spirit of Xu \cite{Xu2001}), $v_x(t)$ for $-w_x'(t)$ and $j_x(t)$ for $J(w_x(t)-Pw_x(t))$. Note that $w_x'(t)$ is defined almost-everywhere and $(w_x(t),-w_x'(t))\in A$ is satisfied almost-everywhere (see \cite{Bar1976}), say both on $[0,\infty)\setminus N_1$ where $N_1$ is a Lebesgue null set.\\

The first step in the proof is to establish $\langle v_x(t),j_x(t)\rangle\geq 0$ and subsequently to establish that $\liminf_{t\to\infty}\langle v_x(t),j_x(t)\rangle=0$. The following results extract from their proof a rate for the $\liminf$ expression. 

\begin{lemma}\label{lem:integralLimInf}
If $f:[0,\infty)\to[0,\infty)$ is Lebesgue integrable with
\[
\int^{\infty}_0f(t)\mathrm{dt}\leq L
\]
for some $L\in [0,\infty)$, then for any Lebesgue null set $N\subseteq [0,\infty)$ and any $k,n$:
\[
\exists t\in [n,\ceil*{L+1}(k+1)+n]\setminus N\left(f(t)\leq\frac{1}{k+1}\right).
\]
\end{lemma}
\begin{proof}
Suppose not. Then there are a Lebesgue null set $N$ and $k,n$ such that for any $t\in [n,\ceil*{L+1}(k+1)+n]\setminus N$ it holds that $f(t)>1/(k+1)$. As $f$ is nonnegative, we get that
\[
\int^{\infty}_0f(t)\mathrm{dt}\geq\int_{[n,(L+1)(k+1)+n]\setminus N}f(t)\mathrm{dt}\geq\frac{((L+1)(k+1)+n-n)}{k+1}=(L+1)
\]
which is a contradiction.
\end{proof}

Now, $\norm{w_x(t)-Pw_x(t)}$ is Lipschitz-continuous as $\norm{x-Px}$ is nonexpansive and $w_x(t)$ is Lipschitz with $\norm{w_x(t)-w_x(s)}\leq 2\norm{v}\vert t-s\vert$ where $v\in Ax$ which exists as $x\in\mathrm{dom}A$ (see the proof of Theorem 1.3 in Chapter III of \cite{Bar1976}). Thus $\norm{w_x(t)-Pw_x(t)}$ is absolutely continuous on every $[0,T]$ which implies that the derivative $\frac{\mathrm{d}}{\mathrm{dt}}\norm{w_x(t)-Pw_x(t)}^2$ exists almost everywhere, say on $[0,\infty)\setminus N_2$, and that this derivative is Lebesgue-integrable such that the fundamental theorem of calculus is valid. Further, as shown in \cite{NR1979}, we have that
\[
\langle v_x(t),j_x(t)\rangle\leq-\frac{1}{2}\frac{\mathrm{d}}{\mathrm{dt}}\norm{w_x(t)-Pw_x(t)}^2
\]
holds almost everywhere, say w.l.o.g.\ also on $[0,\infty)\setminus N_2$ where we assume, also without loss of generality, that $N_2\supseteq N_1$.

\begin{lemma}\label{lem:NRliminf}
Let $b\geq \norm{x-Px}$. For any Lebesgue null set $N\supseteq N_2$ and any $k,n$:
\[
\exists t\in \left[n,\ceil*{\frac{1}{2}b^2+1}(k+1)+n\right]\setminus N\left(\langle v_x(t),j_x(t)\rangle\leq \frac{1}{k+1}\right).
\]
\end{lemma}
\begin{proof}
We have $\langle v_x(t),j_x(t)\rangle\geq 0$ for any $t\in [0,\infty)\setminus N_1$ by accretivity of $A$. As 
\[
\langle v_x(t),j_x(t)\rangle\leq-\frac{1}{2}\frac{\mathrm{d}}{\mathrm{dt}}\norm{w_x(t)-Pw_x(t)}^2
\]
holds almost everywhere, we get 
\begin{align*}
\int^\infty_0\langle v_x(t),j_x(t)\rangle\mathrm{dt}&\leq-\frac{1}{2}\int^\infty_0\frac{\mathrm{d}}{\mathrm{dt}}\norm{w_x(t)-Pw_x(t)}^2\mathrm{dt}\\
&=-\frac{1}{2}\lim_{T\to\infty}(\norm{w_x(T)-Pw_x(T)}^2-\norm{w_x(0)-Pw_x(0)}^2)\\
&=\frac{1}{2}\lim_{T\to\infty}(\norm{w_x(0)-Pw_x(0)}^2-\norm{w_x(T)-Pw_x(T)}^2)\\
&\leq\frac{1}{2}\norm{w_x(0)-Pw_x(0)}^2\\
&\leq\frac{1}{2}b^2.
\end{align*}
By Lemma \ref{lem:integralLimInf}, we get that for any $N\supseteq N_2$ and any $k,n$:
\[
\exists t\in \left[n, \ceil*{\frac{1}{2}b^2+1}(k+1)+n\right]\setminus N\left(\langle v_x(t),j_x(t)\rangle\leq \frac{1}{k+1}\right)
\]
which is the claim.
\end{proof}

The next step in the proof of Nevanlinna and Reich infers the respective $\liminf$ result for the function $\norm{w_x(t)-Pw_x(t)}$ via the convergence condition together with Lemma \ref{lem:NRliminf} and then, using that $\norm{w_x(t)-Pw_x(t)}$ is nonincreasing, infers the convergence of $w_x(t)$. An analysis of this proof yields, in combination with the above, the following quantitative version of Theorem \ref{thm:NR}. For this, we first focus on the special case when $x\in\mathrm{dom}A$. Note that the following theorem does not use the full $\liminf$-rate of the previous lemma but only requires an instantiation of the above for $n=0$.

\begin{theorem}\label{thm:NRquant}
Let $X$ be uniformly convex and uniformly smooth and $A$ be $m$-accretive such that there exists a weak modulus for the convergence condition $\Omega$. Let $\mathcal{S}=\{S(t)\mid t\geq 0\}$ is the semigroup generated by $A$ via the exponential formula. Let $A^{-1}0\neq\emptyset$ with $p\in A^{-1}0$. For any $x\in\mathrm{dom}A$ with $v\in Ax$, we have
\[
\forall k\in\mathbb{N}\ \forall s,s'\geq \chi(\Omega(K,\mathrm{id})(2k+1))\left(\norm{S(s)x-S(s')x}\leq \frac{1}{k+1}\right)
\]
where
\[
\chi(k)=\ceil*{\frac{1}{2}b^2+1}(k+1)
\]
for any $b,K\in\mathbb{N}$ where $b\geq\norm{x-Px}$ as well as $K\geq\max\{\norm{v},\norm{x-p}+\norm{p}\}$.
\end{theorem}
\begin{proof}
First, note that we have
\[
\norm{w_x(t)-p}=\norm{\lim_{n\to\infty}J^n_{t/n}x-p}\leq\norm{x-p}
\]
as $p\in A^{-1}0$ (i.e. $p$ is a fixed point for any resolvent). Therefore
\[
\norm{w_x(t)}\leq\norm{x-p}+\norm{p}
\]
for any $t\in [0,\infty)$. Further, Proposition 1.2 in \cite{Bar1976} implies
\[
\norm{w_x'(t)}\leq\norm{v}
\]
almost everywhere as $v\in Ax$, say for $t\in [0,\infty)\setminus N_3$. w.l.o.g.\ we assume that $N_3\supseteq N_2\supseteq N_1$.

Now, Lemma \ref{lem:NRliminf} yields that for any $k$:
\[
\exists t\in [0,\chi(k)]\setminus N_3\left(\langle v_x(t),j_x(t)\rangle\leq \frac{1}{k+1}\right).\tag{$\dagger$}
\]

Now we choose a sequence $(t_n)\subseteq[0,\infty)\setminus N_3$ using the previous $(\dagger)$ such that $\langle v_x(t_{n}),j_x(t_{n})\rangle\leq \frac{1}{n+1}$ and $t_n\leq\chi(n)$.

This is well-defined as $N_3\supseteq N_2$ and by the above, we have $\norm{w_x(t_n)},\norm{w_x'(t_n)}\leq K$ for all $n$ where also the latter is well-defined. Now, $\mathrm{id}:\mathbb{N}\to\mathbb{N}$ is a rate of convergence for $\langle v_x(t_{n}),j_x(t_{n})\rangle\to 0$. Then by assumption on $\Omega$, we get
\[
\forall k\ \exists n\leq\Omega(K,\mathrm{id})(2k+1)\left(\norm{w_x(t_n)-Pw_x(t_n)}\leq\frac{1}{2(k+1)}\right)
\]
and thus, as $t_n\leq\chi(n)$, we get
\[
\forall k\ \exists t\leq\chi(\Omega(K,\mathrm{id})(2k+1))\left(\norm{w_x(t)-Pw_x(t)}\leq\frac{1}{2(k+1)}\right).
\]
Similar as in \cite{NR1979}, using that
\[
0\leq \langle v_x(t),j_x(t)\rangle\leq-\frac{1}{2}\frac{\mathrm{d}}{\mathrm{dt}}\norm{w_x(t)-Pw_x(t)}^2
\]
almost everywhere, we have that $\norm{w_x(t)-Pw_x(t)}$ is nonincreasing and thus
\[
\forall k\ \forall t\geq\chi(\Omega(K,\mathrm{id})(2k+1))\left(\norm{w_x(t)-Pw_x(t)}\leq\frac{1}{2(k+1)}\right).
\]
We then get
\begin{align*}
\norm{w_x(t)-w_x(t+h)}&\leq\norm{w_x(t)-Pw_x(t)}+\norm{Pw_x(t)-w_x(t+h)}\\
&\leq 2\norm{w_x(t)-Pw_x(t)}
\end{align*}
for all $t,h\geq 0$ similarly as in \cite{NR1979} and therefore
\[
\forall k\ \forall t\geq\chi(\Omega(K,\mathrm{id})(2k+1))\ \forall h\left(\norm{w_x(t)-w_x(t+h)}\leq\frac{1}{k+1}\right)
\]
which yields the claim.
\end{proof}

The following is then an immediate extension to the case of $x\in\overline{\mathrm{dom}A}$.

\begin{theorem}
Assume the conditions of Theorem \ref{thm:NRquant}. Let $x\in\overline{\mathrm{dom}A}$ where $f:\mathbb{N}\to\mathbb{N}$ is such that $f$ is nondecreasing and
\[
\forall n\in\mathbb{N}\ \exists u_n,v_n\in X\left(v_n\in Au_n\land \norm{u_n},\norm{v_n}\leq f(n)\land \norm{u_{n}-x}\leq\frac{1}{n+1}\right).
\]
Then
\[
\forall k\in\mathbb{N}\ \forall s,s'\geq\chi_k(\Omega(K_k,\mathrm{id})(6k+5))\left(\norm{S(s)x-S(s')x}\leq\frac{1}{k+1} \right)
\]
where
\[
\chi_k(j)=\ceil*{\frac{1}{2}b_k^2+1}(j+1)
\]
for any $b_k,K_k\in\mathbb{N}$ where $b_k\geq\norm{x-Px}+\norm{x}+f(3k+2)$ as well as $K_k\geq f(3k+2)+2\norm{p}$.
\end{theorem}
\begin{proof}
By assumption on $f$, we get that there exists a $u,v$ with $v\in Au$ such that $\norm{u},\norm{v}\leq f(3k+2)$ and such that $\norm{x-u}\leq 1/(3(k+1))$. Therefore
\begin{align*}
\norm{S(s)x-S(s')x}&\leq\norm{S(s)x-S(s)u}+\norm{S(s)u-S(s')u}+\norm{S(s')x-S(s')u}\\
&\leq 2\norm{x-u} + \norm{S(s)u-S(s')u}\\
&\leq\frac{2}{3(k+1)} + \norm{S(s)u-S(s')u}.
\end{align*}
Using the previous Theorem \ref{thm:NRquant}, since $v\in Au$, we get that 
\[
\forall k\ \forall s,s'\geq \chi_k(\Omega(K_k,\mathrm{id})(6k+5))\left(\norm{S(s)u-S(s')u}\leq \frac{1}{3(k+1)}\right)
\]
and thus 
\[
\forall k\ \forall s,s'\geq \chi_k(\Omega(K_k,\mathrm{id})(6k+5))\left( \norm{S(s)x-S(s')x}\leq\frac{1}{k+1}\right)
\]
since
\[
\max\{\norm{v},\norm{u-p}+\norm{p}\}\leq \max\{f(3k+2),f(3k+2)+2\norm{p}\}\leq K_k
\]
as well as
\[
\norm{u-Pu}\leq\norm{u-Px}\leq\norm{u-x}+\norm{x-Px}
\]
and thus $\norm{u-Pu}\leq b_k$.
\end{proof}
\begin{remark}\label{rem:genRem}
As revealed by the quantitative analysis, the above result as well as Theorem \ref{thm:NRquant} already hold in general Banach spaces whenever there exist selections of the duality map and of the projection satisfying some simple requirements. See Section \ref{sec:Logic} for further comments on this.
\end{remark}

\subsection{The asymptotic behavior of almost-orbits of nonlinear semigroups}

We now turn to an analysis of Xu's result. For that, consider the setup from Theorem \ref{thm:Xu} and write $\mathcal{S}=\{S(t)\mid t\geq 0\}$ for the semigroup generated by $A$ via the exponential formula as before.

\begin{theorem}\label{thm:Xuquantmeta}
Let $X$ be uniformly convex and uniformly smooth and $A$ be m-accretive such that there exists a weak modulus for the convergence condition $\Omega$. Let $\mathcal{S}=\{S(t) : t\geq 0\}$ be the semigroup generated by $A$ via the exponential formula. Let $A^{-1}0\neq\emptyset$ with $p\in A^{-1}0$ and assume that $P$, the nearest point projection onto $A^{-1}0$, is uniformly continuous on bounded subsets of $X$ with a modulus $\omega:\mathbb{N}^2\to\mathbb{N}$, i.e. 
\[
\forall r,k\in\N\ \forall x,y\in \overline{B}_r(p) \left( \|x-y\|\leq \frac{1}{\omega(r,k)+1} \to \|Px-Py\|\leq \frac{1}{k+1} \right),
\]
and, without loss of generality, assume that $\omega(r,k)\geq k$ for all $r,k\in\mathbb{N}$. Let $u$ be an almost-orbit of $\mathcal{S}$ with a rate of metastability $\Phi$ on the almost-orbit condition, i.e.
\[
\forall k\in\N\ \forall f:\N\to\N\ \exists n\leq \Phi(k,f)\ \forall t\in [0,f(n)] \left( \|S(t)u(n)-u(t+n)\|\leq \frac{1}{k+1} \right).
\]
Let $B\in\mathbb{N}^*$ be such that $\|u(t)-p\|\leq B$ for all $t\geq 0$ and let $f_s:\mathbb{N}\to\mathbb{N}$ for $s\geq 0$ be such that $f_s$ is nondecreasing and
\[
\forall n\in\mathbb{N}\ \exists x_{s,n},y_{s,n}\in X\left(y_{s,n}\in Ax_{s,n}\land \norm{x_{s,n}},\norm{y_{s,n}}\leq f_s(n)\land \norm{x_{s,n}-u(s)}\leq\frac{1}{n+1}\right).
\]
Then we have
\[
\forall k\in\N\ \forall f:\N\to \N\ \exists n\leq \Gamma(k, f)\ \forall t,t'\in [n, n+f(n)] \left( \|u(t)-u(t')\|\leq \frac{1}{k+1} \right),
\]
where
\[
\Gamma(k, f):=\max\{\Gamma'(8k+7,j_{k,f}),\Phi(8k+7, h_{N,f})\mid N\leq \Gamma'(8k+7, j_{k,f})\}
\]
with
\begin{gather*}
h_{N,f}(n):=f(\max\{N,n \}) + \max\{N,n\} - n,\\
j_{k,f}(n):=\max\{n,\Phi(8k+7,h_{n,f})\}-n\\
g_{k,f}(m):=\Omega_{m}(3k+2) + f(m+\Omega_{m}(3k+2)),\\
\Gamma'(k,f):= \Phi(\omega(B, 3k+2), g_{k,f})+\max\{\Omega_{m}(3k+2)\mid m\leq\Phi(\omega(B, 3k+2), g_{k,f})\},
\end{gather*}
for $\Omega_s(k)$ with $s\geq 0$ defined by
\[
\Omega_s(k):=\chi(\Omega(K_{s,k},\id)(3k+2)),
\]
now with
\[
\chi(k):=\ceil*{\frac{1}{2}(B+1)^2+1}(k+1)
\]
for any $K_{s,k}\in\mathbb{N}$ where $K_{s,k}\geq\max\{f_s(\omega(B+1, 3k+2)), B+1+\|p\|\}$.
\end{theorem}

\begin{proof}
For $x\in\mathrm{dom}A$ with $v\in Ax$ consider $S(t)x$. As in the proof of Theorem \ref{thm:NRquant}, we get 
\[
\forall k\in\N\ \forall t\geq\Omega'_{K,b}(k) \left( \|S(t)x-PS(t)x\| \leq \frac{1}{k+1}\right),\tag{$-$}
\]
where 
\[
\Omega'_{K,b}(k):=\ceil*{\frac{1}{2}b^2+1}(\Omega(K,\id)(k)+1),
\]
with $K\geq \max\{\norm{v},\norm{x-p}+\norm{p}\}$ and $b\geq\norm{x-Px}$.\\

\noindent \textbf{Claim 1:} For all $s\geq 0$,
\[
\forall k\in \N\ \forall t\geq \Omega_s(k) \left( \|S(t)u(s)-PS(t)u(s)\| \leq \frac{1}{k+1}\right).
\]
\noindent \textbf{Proof of claim 1:} For given $s\geq 0$, note that by assumption on $f_s$ there exist $y_{s,k}\in Ax_{s,k}$ with $\norm{x_{s,k}},\norm{y_{s,k}}\leq f_s(\omega(B+1, 3k+2))$ such that
\[
\|x_{s,k}-u(s)\|\leq \frac{1}{\omega(B+1, 3k+2)+1}\ \left(\leq \frac{1}{3(k+1)}\right).
\]
For $\chi$ and $K_{s,k}$ as above, since
\[
\|x_{s,k}-Px_{s,k}\|\leq \|x_{s,k}-p\|\leq \|x_{s,k}-u(s)\|+\|u(s)-p\|\leq B+1,
\]
we have by ($-$) that
\[
\forall k\in\N\ \forall t\geq\Omega_s(k) \left( \|S(t)x_{s,k}-PS(t)x_{s,k}\| \leq \frac{1}{3(k+1)}\right),
\]
with $\Omega_s(k)$ defined as above since $\Omega_s(k)=\Omega'_{K_{s,k},(B+1)}(3k+2)$. For $t\geq\Omega_s(k)$, we thus also have
\begin{align*}
\|S(t)u(s)-PS(t)u(s)\|&\leq \|S(t)u(s)-S(t)x_{s,k}\|+\|S(t)x_{s,k}-PS(t)x_{s,k}\|\\
&\qquad\qquad+\|PS(t)x_{s,k}-PS(t)u(s)\|\\
&\leq \|u(s)-x_{s,k}\|+ \|S(t)x_{s,k}-PS(t)x_{s,k}\|\\
&\qquad\qquad+\|PS(t)x_{s,k}-PS(t)u(s)\|\\
&\leq \frac{1}{3(k+1)} + \frac{1}{3(k+1)} + \|PS(t)x_{s,k}-PS(t)u(s)\|.
\end{align*}
Since $\|S(t)x_{s,k}-S(t)u(s)\|\leq \|x_{s,k}-u(s)\|\leq 1/(\omega(B+1, 3k+2)+1)$ (using nonexpansivity of $S(t)$) as well as $\|S(t)x_{s,k}-p\|\leq \|x_{s,k}-p\|\leq B+1$ and $\|S(t)u(s)-p\|\leq B \leq B+1$ (using nonexpansivity of $S(t)$ and that $p$ is a common fixed-point of all $S(t)$), we conclude that $\|PS(t)x_{s,k}-PS(t)u(s)\|\leq 1/(3(k+1))$. This yields the claim.\hfill $\blacksquare$\\

\noindent \textbf{Claim 2:} For all $k\in\N$ and $f:\N\to\N$:
\[
\exists n\leq \Gamma'(k,f)\ \forall t\in [n, n+f(n)] \left( \|u(t)-Pu(t)\|\leq \frac{1}{k+1} \right).
\]
\noindent \textbf{Proof of claim 2:} For given $k\in\N$ and $f:\N\to \N$, consider the function $g_{k,f}$ as defined above. Using the fact that $u$ is an almost-orbit with rate of metastability $\Phi$, there is some $n_0\leq \Phi(\omega(B, 3k+2), g_{k,f})$ such that
\[
\forall t \in [0, g_{k,f}(n_0)] \left( \|S(t)u(n_0)-u(t+n_0)\|\leq \frac{1}{\omega(B, 3k+2)+1} \right).
\]
Since $\|S(t)u(n_0)-p\|, \|u(t+n_0)-p\|\leq B$, we conclude that
\[
\forall t \in [0, g_{k,f}(n_0)] \left( \|PS(t)u(n_0)-Pu(t+n_0)\|\leq \frac{1}{3(k+1)} \right).
\]
Thus, for $t\in[0, g_{k,f}(n_0)]$, we get
\begin{align*}
\|u(t+n_0)-Pu(t+n_0)\|&\leq \|u(t+n_0)-S(t)u(n_0)\|+ \|S(t)u(n_0)-PS(t)u(n_0)\|\\
&\qquad\qquad+ \|PS(t)u(n_0)-Pu(t+n_0)\|\\
&\leq \frac{2}{3(k+1)} + \|S(t)u(n_0)-PS(t)u(n_0)\|.
\end{align*}
Using Claim 1, we get 
\[
\forall t \geq \Omega_{n_0}(3k+2) \left( \|S(t)u(n_0)-PS(t)u(n_0)\| \leq \frac{1}{3(k+1)}\right),
\]
from which follows that
\[
\forall t\in [\Omega_{n_0}(3k+2), g_{k,f}(n_0)] \left( \|u(t+n_0)-Pu(t+n_0)\|\leq \frac{1}{k+1} \right),
\]
and thus
\[
\forall t\in [n_0+\Omega_{n_0}(3k+2), n_0+g_{k,f}(n_0)] \left( \|u(t)-Pu(t)\|\leq \frac{1}{k+1} \right).
\]
This yields the claim by the definition of $g_{k,f}$.\hfill $\blacksquare$\\
	
\noindent \textbf{Claim 3:} For all $k,N\in\N$ and $f:\N\to\N$:
\[
\exists n\in[N,\max\{N,\Phi(2k+1,h_{N,f})\}]\ \forall t\leq f(n) \left( \|S(t)u(n)-u(t+n)\|\leq \frac{1}{k+1} \right).
\]
\textbf{Proof of claim 3:} Since $\Phi$ is a rate of metastability for the almost-orbit $u$, there is $n_0\leq \Phi(2k+1, h_{N,f})$ such that
\[
\forall t\leq h_{N,f}(n_0) \left( \|S(t)u(n_0)-u(t+n_0)\|\leq \frac{1}{2(k+1)} \right),
\]
with $h_{N,f}$ defined as above. Writing $n:=\max\{N,n_0\}\in [N, \max\{N, \Phi(2k+1, h_{N,f})\}]$, we have for $t\leq f(n)$ that
\begin{align*}
\|S(t)u(n)-u(t+n)\|&\leq \|S(t)u(n)-S(t+n-n_0)u(n_0)\|\\
&\qquad\qquad+ \|S(t+n+n_0)u(n_0)-u(t+n)\|\\
&\leq \|u(n)-S(n-n_0)u(n_0)\|\\
&\qquad\qquad+ \|S(t+n-n_0)u(n_0)-u(t+n)\|.
\end{align*}
Since $n-n_0\leq t+n-n_0\leq h_{N,f}(n_0)$, we conclude the claim.\hfill $\blacksquare$\\
	
\noindent \textbf{Claim 4:} For all $k\in\N$ and $f:\N\to\N$, there is some $n_0\leq \Gamma'(8k+7, j_{k,f})$ such that
\[
\exists n_1\leq \max\{n_0, \Phi(8k+7,h_{n_0,f})\}\ \forall t\leq f(n_1) \left( \|u(n_1)-u(t+n_1)\|\leq \frac{1}{2(k+1)} \right).
\]	
\noindent \textbf{Proof of claim 4:} Let $k\in\N$ and $f:\N\to\N$ be given. From Claim 2 with the function $j_{k,f}$ defined as above, we may consider $n_0\leq \Gamma'(8k+7, j_{k,f})$ such that
\[
\forall t\in [n_0, n_0+j_{k,f}(n_0)] \left( \|u(t)-Pu(t)\|\leq \frac{1}{8(k+1)} \right).
\]
By Claim 3, there exists $n_1\in[n_0, \max\{n_0, \Phi(8k+7, h_{n_0,f})\}]$ satisfying
\[
\forall t\leq f(n_1) \left( \|S(t)u(n_1)-u(t+n_1)\|\leq \frac{1}{4(k+1)} \right).
\]
Since $n_1\in [n_0,\max\{n_0, \Phi(8k+7, h_{n_0,f})\}]=[n_0, n_0+j_{k,f}(n_0)]$, we also have $\|u(n_1)-Pu(n_1)\|\leq 1/(8(k+1))$. Thus, for any $t\leq f(n_1)$:
\begin{align*}
\|u(n_1)-u(t+n_1)\|&\leq \|u(n_1)-Pu(n_1)\| + \|Pu(n_1)-S(t)u(n_1)\|\\
&\qquad\qquad+ \|S(t)u(n_1)-u(t+n_1)\|\\
&\leq 2\|u(n_1)-Pu(n_1)\| + \|S(t)u(n_1)-u(t+n_1)\|\\
&\leq \frac{2}{8(k+1)}+\frac{1}{4(k+1)}=\frac{1}{2(k+1)}
\end{align*}
which yields the claim.\hfill $\blacksquare$\\
	
\noindent Lastly, using the $n_1$ from Claim 4, by triangle inequality it follows that
\[
\forall t, t'\in [n_1, n_1+f(n_1)] \left( \|u(t)-u(t')\|\leq \frac{1}{k+1} \right)
\]
and this yields the claim of the theorem, noticing that $n_1\leq\Gamma(k,f)$.
\end{proof}

\begin{remark}
Similar to Remark \ref{rem:genRem}, as revealed by the quantitative analysis, the above result already holds in general Banach spaces whenever there exist suitable selections of the duality map and projection. We again refer to Section \ref{sec:Logic} for further comments on this.
\end{remark}

This theorem is now (essentially) a true finitization of Xu's original convergence result since it trivially (though non-effectively) implies back the original statement but only talks about finite initial segments (if relativized to sequences $t_n$ with $t_n\to\infty$).

\begin{remark}
As used above, if $X$ is uniformly convex, then $P$ is uniformly continuous on bounded subsets of $X$ and it should be noted that in a given modulus of uniform convexity $\eta:(0,2]\to (0,1]$ in the sense that
\[
\forall\varepsilon\in (0,2]\ \forall x,y\in X\left( \norm{x},\norm{y}\leq 1\land \norm{x-y}\geq\varepsilon\rightarrow\norm{\frac{x+y}{2}}\leq 1-\eta(\varepsilon)\right),
\]
one can compute a modulus of uniform continuity $\omega$ for $P$ as used above. Concretely, we want to mention the following result given e.g.\ in \cite{Sch1971}: if $\mathrm{dist}(x,A^{-1}0)\leq r$ and
\[
\norm{x-y}\leq\frac{1}{2}\alpha\left(\frac{\varepsilon}{1+r}\right)
\]
where
\[
\alpha(\varepsilon)=\min\left\{1,\frac{\varepsilon}{4},\frac{\varepsilon\eta(\varepsilon)}{4(1-\eta(\varepsilon))}\right\},
\]
then $\norm{Px-Py}\leq\varepsilon$. From this, a suitable modulus $\omega(r,k)$ can be immediately derived.
\end{remark}

Now, similarly to \cite{KKA2015} and as discussed in the introduction already, the analysis of Xu's result (by being essentially constructive) allows for the extraction of two kinds of quantitative ``translations'' and we now focus on the other variant compared to the above which translates the stronger quantitative assumption of a rate of convergence for the almost-orbit into a rate of convergence of the solution of the Cauchy problem towards a zero of the operator $A$.

\begin{theorem}\label{thm:Xuquantroc}
Let $X$ be uniformly convex and uniformly smooth and $A$ be $m$-accretive such that there exists a weak modulus for the convergence condition $\Omega$. Let $\mathcal{S}=\{S(t)\mid t\geq 0\}$ is the semigroup generated by $A$ via the exponential formula. Let $A^{-1}0\neq\emptyset$ with $p\in A^{-1}0$ and assume that $P$, the nearest point projection onto $A^{-1}0$, is uniformly continuous on bounded subsets of $X$ with a modulus $\omega:\mathbb{N}^2\to\mathbb{N}$, i.e. 
\[
\forall r,k\in\N\ \forall x,y\in \overline{B}_r(p) \left( \|x-y\|\leq \frac{1}{\omega(r,k)+1} \to \|Px-Py\|\leq \frac{1}{k+1} \right),
\]
and, without loss of generality, assume that $\omega(r,k)\geq k$ for all $r,k\in\mathbb{N}$. Let $u$ be an almost-orbit with a rate of convergence $\Phi:\mathbb{N}\to\mathbb{N}$ on the almost-orbit condition, i.e.
\[
\forall k\in\mathbb{N}\ \forall s\geq \Phi(k)\left( \sup_{t\geq 0}\norm{u(s+t)-S(t)u(s)}\leq\frac{1}{k+1}\right).
\]
Let $B\in\mathbb{N}^*$ be such that $\|u(t)-p\|\leq B$ for all $t\geq 0$ and let $f_s:\mathbb{N}\to\mathbb{N}$ for $s\geq 0$ be such that $f$ is nondecreasing and
\[
\forall n\in\mathbb{N}\ \exists x_{s,n},y_{s,n}\in X\left(y_{s,n}\in Ax_{s,n}\land \norm{x_{s,n}},\norm{y_{s,n}}\leq f_s(n)\land \norm{x_{s,n}-u(s)}\leq\frac{1}{n+1}\right).
\]
Then we have
\[
\forall k\ \forall t,t'\geq \max\{\Phi(8k+7),s^*+\Omega_{s^*}(24k+23)\}\left(\norm{u(t)-u(t')}\leq\frac{1}{k+1}\right)
\]
where $s^*=\Phi(\omega(B,24k+23))$ and where $\Omega_s(k)$ is defined as in Theorem \ref{thm:Xuquantmeta}.
\end{theorem}
\begin{proof}
Given a rate of convergence $\Phi$ on the almost-orbit condition, it is clear that $\Phi(k,f):=\Phi(k)$ (ignoring the slight abuse of notation) is a rate of metastability for the almost-orbit. Therefore, by Theorem \ref{thm:Xuquantmeta}, we get that the previously constructed $\Gamma(k,f)$ is rate of metastability for the conclusion. As shown\footnote{Correction to \cite[Proposition 2.6]{KPin2022}: ``$\rho:(0,\infty)\to \mathbb{N}$'' instead of ``$\rho:(0,\infty)\to (0,\infty)$''.} in Proposition 2.6 of \cite{KPin2022}, a function $\rho:(0,\infty)\to \mathbb{N}$ is a Cauchy rate of a sequence iff $\varphi(\varepsilon,f):=\rho(\varepsilon)$ is a rate of metastability (which also holds in our adapted context where we consider rates to be functions operating on natural numbers as errors). Now, using that $\Phi(k,f)=\Phi(k)$, we find by inspection of the defining term that also $\Gamma(k,f)$ is independent of the parameter $f$. Thus, we get that $\Gamma(k):=\Gamma(k,f)$ is a rate of convergence and the given bound in the above theorem just results by simplifying the expressions accordingly.
\end{proof}
Note in particular that the above result is indeed a consequence of the previous metastability result and does not require one to reiterate the proof. In that way, the metastability result already contained the quantitative information regarding rates of convergence. We refer to \cite{Koh2011} for further discussions of such phenomena.

\begin{remark}
Such a rate of convergence of the almost-orbit condition as required as a premise in the above theorem, i.e.\ a $\Phi:\mathbb{N}\to\mathbb{N}$ such that
\[
\forall k\in\mathbb{N}\ \forall s\geq \Phi(k)\left( \sup_{t\geq 0}\norm{u(s+t)-S(t)u(s)}\leq\frac{1}{k+1}\right)
\]
can actually be derived from the seemingly weaker assumption on the existence of a $\Phi$ such that
\[
\forall k\in\mathbb{N}\ \exists s_0\leq \varphi(k)\left( \sup_{t\geq 0}\norm{u(s_0+t)-S(t)u(s_0)}\leq\frac{1}{k+1}\right).
\]
Namely, for given $k\in\mathbb{N}$ and for $s\geq s_0$ with $s_0\leq\varphi(2k+1)$ as stipulated above, we can express $s=s_0+s_1$ for $s_1\geq 0$ and then compute for all $t\geq 0$:
\begin{align*}
\norm{u(s+t)-S(t)u(s)}&=\norm{u(s_0+s_1+t)-S(t)u(s_0+s_1)}\\
&\leq\norm{u(s_0+s_1+t)-S(t+s_1)u(s_0)}\\
&\qquad\qquad+\norm{S(t+s_1)u(s_0)-S(t)u(s_0+s_1)}\\
&\leq\frac{1}{2(k+1)}+\norm{S(t)S(s_1)u(s_0)-S(t)u(s_0+s_1)}\\
&\leq\frac{1}{2(k+1)}+\norm{S(s_1)u(s_0)-u(s_0+s_1)}\\
&\leq\frac{1}{k+1}.
\end{align*}
Thus, $\Phi(k)=\varphi(2k+1)$ is actually a full rate of convergence. This remark also applies to the results regarding rates of convergence presented in \cite{KKA2015} which thus also essentially depend on a rate of convergence of the almost-orbit condition.
\end{remark}

Further, note that in both cases the existence of the assumed bound on $\norm{u(t)-p}$ is actually guaranteed by the assumption of $u$ being an almost-orbit and that $A^{-1}0\neq\emptyset$: the definition implies
\[
\exists s^*\sup_{t\geq 0}\norm{u(t+s^*)-S(t)u(s^*)}\leq 1
\]
and thus for $p\in A^{-1}0$, we have
\begin{align*}
\norm{u(t+s^*)-p}&\leq\norm{u(t+s^*)-S(t)u(s^*)}+\norm{S(t)u(s^*)-p}\\
&\leq 1+\norm{u(s^*)-p}<\infty
\end{align*}
for all $t\geq 0$. As $u$ is continuous, we get that
\[
\sup_{t\in [0,s^*]}\norm{u(t)-p}<\infty
\]
which implies that $u(t)-p$ is bounded in norm. Note that a concrete bound can therefore be computed using a modulus of continuity for $u$ on bounded sets together with a rate of convergence $\Phi$ on the almost-orbit condition and a norm upper bound on $p$.

\section{Logical aspects of the above results}\label{sec:Logic}

As discussed in the introduction already, the above quantitative results and considerations on the convergence condition of Pazy (and its extension by Nevanlinna and Reich) were obtained in the context of the proof mining program (see again the references in the introduction) and as further commented on there, these extractions rest on some logical properties which we want to discuss in this section. For this, we strongly rely on \cite{Koh2008}, in general, and \cite{Pis2022}, in particular, for background information and notation (and only remind on key notions, properties and notations in footnotes).

The starting point for these considerations are the logical systems introduced in \cite{Pis2022} (whose notation we follow) for an axiomatic treatment of set-valued accretive operators and their resolvents which is amenable for the proof-theoretic techniques used in contemporary proof mining for the extraction of explicit quantitative information from (non-)constructive proofs. 

Following \cite{Pis2022}, we denote this system for abstract normed spaces with an accretive operator with total resolvents by $\mathcal{V}^\omega$ which, as discussed in there, allows for the immediate formalization of large parts of m-accretive operator theory. Further, as also shown in \cite{Pis2022}, there is a bound extraction theorem for $\mathcal{V}^\omega$ obtained by following the methods developed in \cite{GeK2008,Koh2005}.

\begin{theorem}[\cite{Pis2022}]\label{thm:metatheoremAccrOp}
Let $\tau$ be admissible\footnote{A type $\tau$ is called admissible if it is of the form $\sigma_1\to(\dots\to(\sigma_k\to X))$ or $\sigma_1\to(\dots\to(\sigma_k\to \mathbb{N}))$, including $\mathbb{N}$ and $X$, and where each $\sigma_i$ is of the form $\mathbb{N}\to(\dots\to(\mathbb{N}\to \mathbb{N}))$ or $\mathbb{N}\to(\dots\to(\mathbb{N}\to X))$ (also including $\mathbb{N}$ and $X$).}, $\delta$ be of the form $\mathbb{N}\to(\dots\to(\mathbb{N}\to \mathbb{N}))$ and $s$ be a closed term of $\mathcal{V}^\omega$ of type $\delta\to\sigma$ for admissible $\sigma$. Let $B_\forall(x,y,z,u)$ / $C_\exists(x,y,z,v)$ be $\forall$-/$\exists$-formulas\footnote{A formula is called a $\forall$-formula (respectively $\exists$-formula) if it has the form $\forall \underline{a}\,F_{qf}(\underline{a})$ (respectively $\exists \underline{a}\,F_0(\underline{a})$), where $F_{qf}$ is quantifier-free and $\underline{a}$ are variables of admissible types.} of $\mathcal{V}^{\omega}$ with only $x,y,z,u$/$x,y,z,v$ free. Let $\Delta$ be a set of sentences of the form\footnote{Here, $\preceq$ is defined pointwise with $x\preceq_X y:=\| x\|_X\le_{\mathbb{R}} \| y\|_X$ for the base type $X$ and $n\preceq_X m:=n\leq m$ for the base type $\mathbb{N}$.}  $\forall\underline{a}^{\underline{\delta}} \exists \underline{b}
\preceq_{\underline{\sigma}}
\underline{r}\underline{a}\forall\underline{c}^{\underline{\gamma}}
B_{qf}(\underline{a},\underline{b},\underline{c})$ where $B_{qf}$ is quantifier-free, $\underline{r}$ is a tuple of closed terms of suitable types and all types in $\underline{\delta},\underline{\sigma},\underline{\gamma}$ are admissible

If 
\[
\mathcal{V}^{\omega}+\Delta\vdash\forall x^\delta\ \forall y\preceq_\sigma s(x)\ \forall z^\tau\left(\forall u^{\mathbb{N}} B_\forall(x,y,z,u)\rightarrow\exists v^{\mathbb{N}} C_\exists(x,y,z,v)\right), 
\]
then one can extract a bar-recursively computable partial function $\Phi:S_{\delta}\times S_{\widehat{\tau}}\times\mathbb{N}\rightarrow\mathbb{N}$ such that for all $x\in S_\delta$, $z\in S_\tau$, $z^*\in S_{\widehat{\tau}}$ and all $n\in\mathbb{N}$, if $z^*\gtrsim z$\footnote{Here and in the following, $z^*\gtrsim z$ denotes the extended strong majorization relation from the definition of the model $\mathcal{M}^{\omega,X}$ as defined in \cite{GeK2008} (see also \cite{Pis2022}).} and $n\geq_\mathbb{R}\norm{J^A_{1}(0)}_X$, then
\[
\mathcal{S}^{\omega,X}\models\forall y\preceq_\sigma s(x)\left(\forall u\leq\Phi(x,z^*,n)\, B_\forall(x,y,z,u)\rightarrow\exists v\leq\Phi(x,z^*,n)\,C_\exists(x,y,z,v)\right)
\]
holds for all (real) normed spaces $(X,\norm{\cdot})$ with $\chi_A$ interpreted by the characteristic function of an m-accretive operator $A$ and $J^{\chi_A}$ by corresponding resolvents $J^A_\gamma$ for $\gamma>0$ whenever $\mathcal{S}^{\omega,X}\models\Delta$.
\end{theorem}

If we remove (essentially) the law of excluded middle and move to a semi-constructive version of $\mathcal{V}^\omega$, then the above bound extraction theorem can be vastly strengthened, like, e.g., by lifting the restrictions on the types, on the extensionality as well as on the complexity of the formula, on the amount of choice and to some degree on the complexity of the non-constructive principles allowed (see \cite{GeK2006} for an extensive discussion).

In that vein, we consider the system $\mathcal{V}^\omega_i$ which is defined (essentially similar to the definition of the related system $\mathcal{T}^\omega_i$ given in \cite{KP2022}) by extending the system $\mathcal{A}^{\omega}_i[X,\norm{\cdot}]=\mathrm{E}\text{-}\mathrm{HA}^{\omega}[X,\norm{\cdot}]+\mathrm{AC}$ from \cite{GeK2006} with the axioms of $A$ and its resolvent $J^{A}_{\gamma}$ as outlined in \cite{Pis2022}
 
Now, by use of the same methods developed in \cite{GeK2006} (see also \cite{KP2022}), we obtain the following bound extraction result.
\begin{theorem}[essentially \cite{KP2022}]\label{thm:metatheoremAccrOpI}
Let $\delta$ be of the form $\mathbb{N}\to(\dots\to(\mathbb{N}\to \mathbb{N}))$ and $\sigma,\tau$ be arbitrary, $s$ be a closed term of suitable type. Let $B(x,y,z)$ / $C(x,y,z,u)$ be arbitrary formulas of $\mathcal{V}^{\omega}_i$ with only $x,y,z$/$x,y,z,u$ free. Let $\Gamma_\neg$ be a set of sentences of the form $\forall\underline{u}^{\underline{\zeta}}(C(\underline{u})\rightarrow \exists\underline{v}\preceq_{\underline{\beta}}\underline{t}\underline{u}\neg D(\underline{u},\underline{v}))$ where $\underline{\zeta},\underline{\beta}$ are arbitrary types, $C,D$ are arbitrary formulas and $\underline{t}$ are closed terms.

If
\[
\mathcal{V}^{\omega}_i+\Gamma_\neg\vdash\forall x^\delta\ \forall y\preceq_\sigma s(x)\ \forall z^\tau\,(\neg B(x,y,z)\rightarrow\exists u^\mathbb{N}C(x,y,z,u)), 
\]
one can extract a $\Phi:\mathcal{S}_\delta\times\mathcal{S}_{\hat{\tau}}\times\mathbb{N}\to\mathbb{N}$ with is primitive recursive in the sense of G\"odel such that for any $x\in\mathcal{S}_\delta$, any $y\in\mathcal{S}_\sigma$ with $y\preceq_\sigma s(x)$, any $z\in\mathcal{S}_{\tau}$ and $z^*\in\mathcal{S}_{\hat{\tau}}$ with $z^*\gtrsim z$\footnote{Here, $\gtrsim$ denotes (not necessarily strong) majorization interpreted in the model $\mathcal{S}^{\omega,X}$.} and any $n\in\mathbb{N}$ with $n\geq_\mathbb{R}\norm{J^A_{1}(0)}_X$, we have that
\[
\mathcal{S}^{\omega,X}\models\exists u\leq\Phi(x,z^*,n)\,(\neg B(x,y,z)\rightarrow C(x,y,z,u))
\]
holds for all (real) normed spaces $(X,\norm{\cdot})$ with $\chi_A$ interpreted by the characteristic function of an m-accretive operator $A$ and $J^{\chi_A}$ by corresponding resolvents $J^A_\gamma$ for $\gamma>0$ whenever $\mathcal{S}^{\omega,X}\models\Gamma_\neg$.
\end{theorem}

Now, while $\mathcal{V}^\omega$/$\mathcal{V}^\omega_i$ are strong base systems (see again the discussion in \cite{Pis2022}), the applications presented here require the use of further extensions of these systems to handle the various additional notions present in the theorems of Nevanlinna and Reich as well as Xu. We shortly mention these in the following, before turning to logical remarks on the convergence condition. 

A point neglected in the following discussion is the treatment of the main object of the above results and proofs: the semigroup generated by an accretive operator via the exponential formula. A detailed treatment of those (together with various other objects surrounding them), is given in the forthcoming \cite{Pis2022b} and the combination of the logical considerations of the following subsection and of \cite{Pis2022b} then provides the full underlying system for formalizing the above results as well as their proofs and thus provide the basis for the extractions outlined above.

\subsection{Uniform convexity and projections}

As discussed already in some of the earliest papers on the treatment of abstract spaces in proof mining (see e.g.\ \cite{GeK2008}), uniformly convex spaces can be treated by adding an additional constant together with a corresponding universal axiom to express that this new constant represents a modulus of uniform convexity.

In the works \cite{NR1979,Xu2001}, the uniform convexity is only assumed to infer the existence of an (in the case of Xu, uniformly continuous) selection of the projection map onto closed and convex subsets of $X$. In fact, the only selection map of a projection ever needed is a selection of the projection onto the set $A^{-1}0$ which we as before denote just by $P$. For that, the set $A^{-1}0$ is assumed to be non-empty which can be hardwired into the language of the systems by adding a designated constant $p_0$ of type $X$ together with the corresponding axiom
\begin{enumerate}
\item[($NE$)] $0\in Ap_0$.
\end{enumerate}
In the context of the above systems for the treatment of m-accretive operators and their extensions, this kind of projection map can be immediately treated by adding a further constant $P$ of type $X(X)$ together with the axiom scheme
\begin{enumerate}
\item[($P1$)] $\forall x^X,p^X\left(0\in A(Px)\land \left(0\in Ap\rightarrow\norm{x-Px}\leq_\mathbb{R}\norm{x-p}\right)\right)$,
\end{enumerate}
characterizing that $P$ is indeed a selection of the projection onto the set $A^{-1}0$. In particular, note also that these axioms are in particular purely universal as the statement $0\in Ap$ in the context of the system $\mathcal{V}^\omega$ is quantifier-free, being an abbreviation for $\chi_A(p,0)=_\mathbb{N}0$ (see the discussion in \cite{Pis2022}). Note also again that in that way, as stressed before, the treatment of the projection does not require it to be unique but only to be a suitable selection from the potentially multi-valued nearest point projection.

Further, it is immediate from the axioms that $P$ is provably majorizable in $\mathcal{V}^\omega+(P1)+(NE)$ as we can prove 
\[
\norm{Px}\leq\norm{x}+\norm{x-Px}\leq\norm{x}+\norm{x-p_0}\leq 2\norm{x}+\norm{p_0}
\]
from the axioms ($P1$) and ($NE$).\\

If extensionality or continuity is needed for the projection $P$ (as is the case in the context of Xu's result), the above system needs to be extended with a modulus of uniform continuity $\omega^P$ of type $\mathbb{N}\to (\mathbb{N}\to\mathbb{N})$ together with a corresponding axiom like
\begin{enumerate}
\item[($P2$)] $\begin{cases}\forall r^\mathbb{N},k^\mathbb{N},x^X,y^X \Big( \norm{x-p_0},\norm{y-p_0}\leq r\\
\qquad\qquad \land\ \|x-y\|\leq \frac{1}{\omega^P(r,k)+1} \to \|Px-Py\|\leq \frac{1}{k+1} \Big).\end{cases}$
\end{enumerate}

In that way, the bound extraction theorems stated in Theorem \ref{thm:metatheoremAccrOp} and Theorem \ref{thm:metatheoremAccrOpI} immediately extend to the system $\mathcal{V}^\omega+(NE)+(P1)$ $(+(P2))$ where one then additionally requires $n$ to satisfy $n\geq\norm{p_0}$ (and in the case of $(P2)$, $\Phi$ additionally depends on $\omega^P$).

\subsection{Uniform smoothness and the normalized duality map}

Regarding uniformly smooth spaces, we focus on the dual characterization of such spaces via the requirement of a single-valued duality map $J$ which is norm-to-norm uniformly continuous on bounded subsets.

As becomes clear through inspection of the analyses presented above, they actually `only' require a function $J:X\to X^*$ which selects some point from the duality set and satisfies the following properties:
\begin{enumerate}
\item $\langle y-Px,J(x-Px)\rangle\leq 0$ for all $x\in X$ and all $y\in A^{-1}0$;
\item $\langle u-v,J(x-y)\rangle\geq 0$ for all $(x,u),(y,v)\in A$.
\end{enumerate}
Both properties are satisfied for the unique selection if $X$ is uniformly smooth and $A$ is m-accretive with $A^{-1}0\neq\emptyset$ and where a selection $P$ of the projection onto that set exists as above. But, actually, \emph{any} such selection suffices which is in particular suggested by the proof-theoretic perspective. 

In that way, our situation is similar to that of \cite{KL2012} where the authors introduced a proof-theoretic treatment of such duality selection maps in the context of the general framework introduced in \cite{GeK2008,Koh2005} and we in the following build on this treatment (and in that way strongly rely on the background and notation of \cite{KL2012}) to provide further extensions of the above system to deal with these objects associated with the duality mapping.

In that way, we find that the use of the duality map made in the above extractions can be formalized in the context of the extension of the above system $\mathcal{V}^\omega+(P1)+(NE)$ (+$(P2)$) by the constants $J$ and $\omega^J$ together with the axioms introduced in \cite{KL2012}, i.e.
\begin{enumerate}
\item[$(J1)$] $\begin{cases}\forall x^X,y^X\Big(Jxx=_\mathbb{R}\norm{x}^2\land\vert Jxy\vert\leq_\mathbb{R}\norm{x}\norm{y}\\
\qquad\qquad\land\ \forall\alpha^1,\beta^1,u^X,v^X\left(Jx(\alpha u+\beta v)=_\mathbb{R}\alpha Jxu+\beta Jxv\right)\Big),\end{cases}$
\item[$(J2)$] $\begin{cases}\forall x^X, y^X, z^X, k^\mathbb{N},R^\mathbb{N}\Big(\norm{x},\norm{y}<_\mathbb{R}R\\
\qquad\qquad\land\ \norm{x-y}<_\mathbb{R}\frac{1}{\omega^J(k,R)+1}\rightarrow\vert Jxz-Jyz\vert\leq_\mathbb{R}\frac{\norm{z}}{k+1}\Big),\end{cases}$
\end{enumerate}
as well as two additional axioms expressing the above properties (i) and (ii)
\begin{enumerate}
\item[$(M1)$] $\forall x^X,y^X\left( 0\in Ay\rightarrow \langle y-Px,J(x-Px)\rangle\leq 0\right),$
\item[$(M2)$] $\forall x^X,y^X,u^X,v^X\left( u\in Ax\land v\in Ay\rightarrow \langle u-v,J(x-y)\rangle\geq 0\right).$
\end{enumerate}
The bound extraction theorems stated before also here immediately extend to the system $\mathcal{V}^\omega+(P1)+(NE)+(J1)+(M1)+(M2)$ ($+(P2)+(J2)$) as is immediately clear through the discussion in \cite{KL2012} and the fact that all the new axioms are universal.\\

Now, by itself, the existence of a selection functional for the duality map in particular does not imply that the latter has to be single-valued. However, as shown by K\"ornlein \cite{Koe2015}, the existence of a selection functional with is uniformly norm-to-norm continuous (i.e.\ satisfies axioms $(J2)$) is actually equivalent to uniform smoothness of the space $X$ and thus actually implies that this selection is the unique selection. As discussed in \cite{KL2012}, the uniform continuity of the selection is already implied by the logical methodology in the case that the proof relies on the extensionality of it.

However, as the above analysis shows, the proofs of Nevanlinna and Reich as well as Xu do not rely on uniform continuity or even extensionality of $J$ and, in that way, can already be formalized in the system $\mathcal{V}^\omega+(P1)+(NE)+(J1)+(M1)+(M2)$ (modulo the treatment of semigroups from \cite{Pis2022b} and with $(P2)$ in the case of Xu) which also explains the absence of any such moduli $\omega^J$ in the analysis. In particular this additionally shows that the results are already valid in the context of the existence of a selection functional satisfying (i) and (ii) which is potentially weaker than uniform smoothness.

This insight that conditions (i) and (ii) are sufficient, now here facilitated via a proof-theoretic method, was essentially already observed in the last section of the work of Nevanlinna and Reich \cite{NR1979} although it was not clearly stated. Instead, they listed additional conditions on the operator in order to guarantee that the conditions (i) and (ii) are naturally satisfied. Concretely, they require that the operator then is accretive in the sense of Browder \cite{Bro1967} to enable that the condition (ii) is satisfied for any possible selection $J$ and they require that $A^{-1}0$ is a so-called proximal sun (see \cite{NR1979}) in order to guarantee that a selection satisfying (i) always exists and they require that the semigroup is differentiable so that the orbit generated by the Crandall-Liggett formula is actually a solution of the corresponding initial valued problem (as shown in \cite{CL1971b}). In the vein of the previous logical discussion, we thus find that our above quantitative results also apply to these generalizations. 

\subsection{Logical aspects of the convergence condition}

The main underlying technical tools of the bound extraction theorems mentioned above are, on the one hand, the use of a combination of the Dialectica (or functional) interpretation of G\"odel \cite{Goe1958} with (a modified version of) the notion of majorizability due to Howard \cite{How1973} together with a negative translation (see \cite{Kur1951}). The approach via this combination is due to Kohlenbach \cite{Koh1992} and in its modern form, with the additional abstract types, was introduced in \cite{GeK2008,Koh2005}. 

On the other hand, in the semi-constructive cases, we rely on a combination of the above mentioned notion of majorizability with the modified realizability interpretation due to Kreisel \cite{Kre1959,Kre1962}, a combination which is originally due to Kohlenbach \cite{Koh1998} and in its modern form, with the additional abstract types, was introduced in \cite{GeK2006}.\\

Besides the quantitative analyses of the results of Nevanlinna and Reich as well as Xu, the main contribution of this paper is the introduction of the new notions of ``moduli for the convergence condition''. Already in Bishop's work \cite{Bis1970}, arguments for the functional interpretation as the correct numerical interpretation of theorems of the form $\exists\forall\rightarrow\exists\forall$ are given and, in modern times, the proof mining program has been very effective in arguing that the monotone functional interpretation (in combination with a negative translation) provides the right numerical information in the search for uniform bounds in analysis (see in particular the detailed discussion in \cite{KO2003}). In the following, we will now see how, through this lens, these moduli actually arise from the underlying logical methodology and thus, in various ways, represent the real finitary core of the convergence condition from both a classical and a constructive perspective.

\subsubsection{The convergence condition from a classical perspective}

Based on the equivalence laid out in Lemma \ref{lem:CCequiv}, any proof that a class of operators satisfies the convergence condition, written in the immediate formal translation 
\begin{align*}
&\forall (x_n)^{\mathbb{N}\to X}, (y_n)^{\mathbb{N}\to X}, K^0\Bigg(\forall i^\mathbb{N}\left( y_i\in Ax_i\land\norm{x_i},\norm{y_i}\leq_\mathbb{R} K\right)\\
&\qquad\land\forall a^\mathbb{N}\ \exists b^\mathbb{N}\ \forall c^\mathbb{N}\left( c\geq_\mathbb{N} b\rightarrow\vert\langle y_c,J(x_c-Px_c)\rangle\vert\leq_\mathbb{R}\frac{1}{a+1}\right)\tag{$1$}\\
&\qquad\qquad\rightarrow\forall k^\mathbb{N},N^\mathbb{N}\ \exists n^\mathbb{N}\left(n\geq_\mathbb{N} N\land\norm{x_n-Px_n}\leq_\mathbb{R}\frac{1}{k+1}\right)\Bigg),
\end{align*}
can be transformed into a proof for satisfying the equivalent statement
\begin{align*}
&\forall x^X,y^X,K^\mathbb{N},k^\mathbb{N}\ \exists n^\mathbb{N}\bigg(y\in Ax\land \|x\|, \|y\|\leq_\mathbb{R} K\\
&\qquad\qquad\qquad\land\ \vert\langle y, J(x-Px)\rangle\vert \leq_\mathbb{R} \frac{1}{n+1} \to \|x-Px\|\leq_\mathbb{R} \frac{1}{k+1} \bigg),\tag{$2$}
\end{align*}
however at the expense of using classical logic as well as countable choice. However, this use of countable choice is in essence only applied to a quantifier-free formula and thus is an instance of $\mathrm{QF}\text{-}\mathrm{AC}$. It is clear that (after equivalently writing ($2$) with $<_\mathbb{R}$ used in the conclusion to make the inner matrix existential) the negative translation of ($2$) is equivalent to its original version by the use of Markov's principle and thus that the negative translation followed by the monotone functional interpretation, applied to ($2$), immediately produces a full modulus (as defined in Definition \ref{def:CCmodFull}) as the suggested finitization of this variant of the convergence condition.

Thus, a priori, through the application of the classical metatheorem given in Theorem \ref{thm:metatheoremAccrOp}, we have the following:
\begin{proposition}
There are primitive-recursive (in the sense of G\"odel) translations which transform any full modulus for the convergence condition into a solution of the negative translation followed by the monotone functional interpretation of \emph{($1$)}, and vice versa.
\end{proposition}
Therefore, the two variants of the convergence condition and the accompanying moduli can be extracted from proofs and used interchangeably without yielding a far increase of complexity beyond the principles used in the proof. Thus the bound extraction result discussed in Theorem \ref{thm:metatheoremAccrOp} guarantees the extractability of such moduli even from classical proofs that $A$ satisfies the convergence condition, provided that the proof can be formalized in $\mathcal{V}^\omega+(P1)+(P2)+(NE)+(J1)+(M1)+(M2)+\Delta$ (which we abbreviate in the following by $\mathcal{C}^\omega$) for suitable $\Delta$ or any extension/fragment thereof pertaining to the bound extraction theorems. Even further, as already hinted on in Remark \ref{rem:strictSNEremark}, this extraction is already possible from suitable proofs of the much weaker requirement
\[
\forall (x,y)\in A\, \left( \langle y, J(x-Px)\rangle =0 \to \|x-Px\|=0 \right).
\]
In that way, a formalized version of the argument in Remark \ref{rem:strictSNEremark} in fact shows the following:
\begin{proposition}\label{pro:fullExtract}
If $\mathcal{C}^\omega$ (or any suitable extension or fragment thereof) proves that $A$ satisfies 
\[
\forall (x,y)\in A\, \left( \langle y, J(x-Px)\rangle =0 \to \|x-Px\|=0 \right),
\]
then from the proof one can extract a (potentially bar-recursively) computable full modulus for the convergence condition. If the proof does not use choice, then the modulus is even primitive recursive in the sense of G\"odel.
\end{proposition}
This in particular also holds if there exists a suitable proof of the convergence condition itself as this proof can be transformed into a proof of the above property (without any additional use of classical logic or choice).\\

However, the modularity of the approach to quantitative information via the monotone functional interpretation further yields that from any proof using the convergence condition as a premise (formulated in any variant as discussed above) and formalizable in the respective systems, quantitative information on the conclusion can be extracted which depends then additionally on such a modulus solving the monotone functional interpretation of the convergence condition. This is collected in the following derived metatheorem:

\begin{theorem}\label{thm:macro}
Under the assumptions of Theorem \ref{thm:metatheoremAccrOp} we have 
the following: If
\[
\mathcal{C}^\omega\vdash\forall x^\delta\ \forall y\preceq_\sigma s(x)\ \forall z^\tau\left( A \ \mbox{satisfies the convergence condition}\ \rightarrow\exists v^{\mathbb{N}} C_\exists(x,y,z,v)\right),
\]
then one can extract a bar-recursively computable partial function $\Phi:S_{\delta}\times S_{\widehat{\tau}}\times (S_{\mathbb{N}\to (\mathbb{N}\to \mathbb{N})})^2\times \mathbb{N}\rightarrow\mathbb{N}$ such that for all $x\in S_\delta$, $z\in S_\tau$, $z^*\in S_{\widehat{\tau}}$, $\Omega^f,\omega\in S_{\mathbb{N}\to (\mathbb{N}\to \mathbb{N})}$ and all $n\in\mathbb{N}$, if $z^*\gtrsim z$ and $n\geq_\mathbb{R}\norm{J^A_1(0)}_X,\norm{p_0}_X$ as well as $\omega\gtrsim\omega^P$, then
\begin{align*}
\mathcal{S}^{\omega,X}\models&\forall y\preceq_\sigma s(x)\Big( \Omega^f\text{ is a full modulus for the convergence condition for } A \\
&\qquad\qquad\rightarrow\exists v\leq_\mathbb{N}\Phi(x,z^*,\Omega^f,\omega,n)\,C_\exists(x,y,z,v)\Big)
\end{align*}
holds for all (real) normed spaces $(X,\norm{\cdot})$ with $\chi_A$ interpreted by the characteristic function of an m-accretive operator $A$ and $J^{\chi_A}$ by corresponding resolvents $J^A_\gamma$ for $\gamma>0$ whenever $\mathcal{S}^{\omega,X}\models\Delta$.

Moreover: if the proof does not use choice, then the modulus is even primitive recursive in the sense of G\"odel. The result remains true for any suitable extension or fragment of $\mathcal{C}^\omega$.
\end{theorem}

In that way, by Proposition \ref{pro:fullExtract} and Theorem \ref{thm:macro}, we find that a full modulus is indeed the right quantitative notion for the convergence condition in the sense that both items (i) and (ii), discussed previously as the central properties before Definition \ref{def:CCmodFull}, are fulfilled.

\subsubsection{The convergence condition from a constructive perspective}

From the semi-constructive perspective of the monotone modified realizability interpretation and the associated systems $\mathcal{V}^{\omega}_i+\Gamma_\neg$ and their extensions, the quantitative version of the convergence condition is exactly what is captured by the notion of the `simple' modulus introduced in Definition \ref{def:CCmod}. Concretely, applying the monotone modified realizability interpretation to the formal statement ($1$) considered previously, we get that it asks for a functional $\Omega$ which transforms $K$ and majorants for $(x_n),(y_n)$ (which w.l.o.g.\ are assumed to coincide with the constant $K$-function and are in that way represented by the input $K$) and a majorant of a realizer for the premise $\forall a\ \exists b\ \forall c\left( c\geq b\rightarrow\vert\langle y_c,J(x_c-Px_c)\rangle\vert\leq 1/(a+1)\right)$, i.e.\ of a $\varphi$ of type $1$ such that
\[
\forall a, c\left( c\geq \varphi(a)\rightarrow\vert\langle y_c,J(x_c-Px_c)\rangle\vert\leq\frac{1}{a+1}\right)
\]
into a majorant of a realizer for the conclusion $\forall k,N\ \exists n\left(n\geq N\land\norm{x_n-Px_n}\leq 1/(k+1)\right)$, i.e.\ into an $\Omega(K,\varphi)$ of type $\mathbb{N}\to (\mathbb{N}\to \mathbb{N})$ such that
\[
\forall k,N\exists n\leq \Omega(K,\varphi)(k,N)\left(n\geq N\land\norm{x_{n}-Px_{n}}\leq\frac{1}{k+1}\right).
\]
Thus, this is exactly what is represented by a `simple' modulus for the convergence condition. An immediate application of the bound extraction result contained in Theorem \ref{thm:metatheoremAccrOpI} yields the following result, similarly to the previous Proposition \ref{pro:fullExtract}. For this, we now work over the semi-constructive variant of the previous theories. Concretely, we abbreviate with $\mathcal{C}^\omega_i$ in the following $\mathcal{V}^\omega_i+(P1)+(NE)+(J1)+(M1)+(M2)+\Gamma_\neg$ for suitable $\Gamma_\neg$.
\begin{proposition}
If $\mathcal{C}^\omega_i$ (or any suitable extension or fragment thereof) proves that $A$ satisfies
\[
\forall (x,y)\in A\, \left( \langle y, J(x-Px)\rangle =0 \to \|x-Px\|=0 \right),
\]
then from the proof one can extract a primitive-recursive full modulus for the convergence condition.
\end{proposition}
As discussed before, this in particular also holds if there exists a suitable proof of the convergence condition.\\

Note that in the presence of the previous Proposition \ref{pro:fullExtract}, the above result in nevertheless not void. While an intuitionistic proof is especially a classical proof, Proposition \ref{pro:fullExtract} of course guarantees already the extractability of a full modulus. However, this only applies in the case that the additional axioms $\Gamma_\neg$ potentially contained in the above system $\mathcal{C}^\omega_i$ are essentially of type $\Delta$ as required by Proposition \ref{pro:fullExtract}. So if the real strength of $\Gamma_\neg$ is used while restricting to intuitionistic logic, then the above result nevertheless guarantees the existence and extractability of a primitive recursive full modulus.\\

Now, in similarity to Theorem \ref{thm:macro}, we obtain a macro for the logical metatheorem contained in Theorem \ref{thm:metatheoremAccrOpI} which guarantees that now from a semi-constructive proof of a result using the convergence condition as a premise, one can extract a transformation which transforms any modulus for the convergence condition into information on the conclusion, even in the presence of the axioms $\Gamma_\neg$.

\begin{theorem}
Under the assumptions of Theorem \ref{thm:metatheoremAccrOpI} we have 
the following: If
\[
\mathcal{C}^\omega_i\vdash\forall x^\delta\ \forall y\preceq_\sigma s(x)\ \forall z^\tau\left( A \ \mbox{satisfies the convergence condition}\ \rightarrow \exists u^\mathbb{N}C(x,y,z,u)\right),
\]
one can extract a $\Phi:S_{\delta}\times S_{\widehat{\tau}}\times S_{\mathbb{N}\to ((\mathbb{N}\to\mathbb{N})\to (\mathbb{N}\to (\mathbb{N}\to\mathbb{N})))}\times S_{\mathbb{N}\to (\mathbb{N}\to \mathbb{N})}\times \mathbb{N}\rightarrow\mathbb{N}$ with is primitive recursive in the sense of G\"odel such that for any $x\in S_\delta$, any $y\in S_\sigma$ with $y\preceq_\sigma s(x)$, any $z\in S_{\tau}$ and $z^*\in S_{\hat{\tau}}$ with $z^*\gtrsim z$\footnote{Here, $\gtrsim$ denotes (not necessarily strong) majorization interpreted in the model $\mathcal{S}^{\omega,X}$, as before.}, any $\Omega\in S_{\mathbb{N}\to ((\mathbb{N}\to\mathbb{N})\to (\mathbb{N}\to (\mathbb{N}\to\mathbb{N})))}$ and any $n\in\mathbb{N}$, $\omega\in S_{\mathbb{N}\to (\mathbb{N}\to \mathbb{N})}$ with $n\geq_\mathbb{R}\norm{J^A_{1}(0)}_X,\norm{p_0}_X$ and $\omega\gtrsim\omega^P$, we have that
\begin{align*}
\mathcal{S}^{\omega,X}\models\exists u\leq\Phi(x,z^*,\Omega,\omega,n)&\Big( \Omega\text{ is a modulus for the convergence condition for } A \\
&\qquad\rightarrow C(x,y,z,u)\Big)
\end{align*}
holds for all (real) normed spaces $(X,\norm{\cdot})$ with $\chi_A$ interpreted by the characteristic function of an m-accretive operator $A$ and $J^{\chi_A}$ by corresponding resolvents $J^A_\gamma$ for $\gamma>0$ whenever $\mathcal{S}^{\omega,X}\models\Gamma_\neg$.
\end{theorem}

Note lastly that it is also this result which a priori guaranteed the dependence of the quantitative versions of the result of Nevanlinna and Reich as well as Xu on our `simple' modulus instead of on the full modulus and which in that way lies behind the extraction.\\

\noindent
{\bf Acknowledgment:} We want to thank Ulrich Kohlenbach for the insightful comments and discussions that benefited this paper. The results of this paper form the main part of Chapter 5 in the doctoral dissertation of the second author. Both authors were supported by the `Deutsche Forschungs\-gemein\-schaft' Project DFG KO 1737/6-2.

\bibliographystyle{plain}
\bibliography{ref}

\end{document}